\documentclass[final]{siamart171218}

\usepackage{mathrsfs}
\usepackage{epsfig, graphicx}
\usepackage{latexsym,amsfonts,amsbsy,amssymb}
\usepackage{amsmath}
\usepackage{color}
\usepackage{mathtools}
\usepackage{float}
\usepackage{cite}
\usepackage{algorithm,algpseudocode}
\usepackage{amsmath}
\usepackage{amsfonts}
\usepackage{amssymb}
\usepackage{tikz}
\usepackage{pgfplots}
\usepackage{pgfplotstable}
\usepackage{comment}
\usepackage{caption}  
\pgfplotsset{compat=1.14}

\usepackage{listings}
\definecolor{listinggray}{gray}{0.9}

\newcommand{\HANSO}{\textsf{HANSO}}
\newcommand{\ROBOBOA}{\textsf{ROBOBOA}}
\newcommand{\Ntheta}{N_{\theta}} 
\newcommand{\R}{\mathbb{R}} 
\newcommand{\Rp}{\mathbb{R}_+} 
\newcommand{\pb}{\boldsymbol{p}}
\newcommand{\xb}{\boldsymbol{x}}
\newcommand{\yb}{\boldsymbol{y}}
\newcommand{\thetab}{\boldsymbol{\theta}}
\newcommand{\Tlow}{T^{{\rm low}}}
\newcommand{\cD}{\mathcal{D}}
\newcommand{\cU}{\mathcal{U}}
\newcommand{\cB}{\mathcal{B}}

\newcommand{\revise}[1]{#1}
\graphicspath{{images/}}

\title{Tuning Multigrid Methods with Robust Optimization and Local
  Fourier Analysis\thanks{Submitted to the editors DATE.
\funding{The work of J.B.\ was partially funded by U.S. Department of Energy, Office of
Science, Office of Advanced Scientific Computing Research under Award
Number DE-SC0016140.  The work of S.M.\ was partially funded by an NSERC Discovery Grant.
The work of M.M.\ and S.M.W.\ was
supported by the U.S.\ Department of Energy,
Office of Science, Office of Advanced Scientific Computing Research, applied
mathematics program under contract no.\ DE-AC02-06CH11357.}}}

\author{Jed Brown\footnotemark[2]
\and Yunhui He\footnotemark[3]\ \footnotemark[4]
  \and Scott MacLachlan\footnotemark[4]
  \and Matt Menickelly\footnotemark[5]
  \and Stefan M.\ Wild\footnotemark[5]}

\newsiamremark{remark}{Remark}

\begin{document}


\maketitle
\renewcommand{\thefootnote}{\fnsymbol{footnote}}
\footnotetext[2]{Department of Computer Science,
University of Colorado Boulder,
430 UCB, Boulder, CO 80309, USA
  (\email{jed@jedbrown.org}).}
\footnotetext[3]{Department of Applied Mathematics, University of Waterloo, Waterloo, ON N2L 3G1, Canada (\email{yunhui.he@uwaterloo.ca}).}
\footnotetext[4]{Department of Mathematics and Statistics, Memorial University of Newfoundland, St. John's, NL A1C 5S7, Canada
  (\email{yunhui.he@mun.ca}, \email{smaclachlan@mun.ca}).}
\footnotetext[5]{Mathematics and Computer Science Division, Argonne National Laboratory, Lemont, IL 60439,
USA
  (\email{mmenickelly@anl.gov}, \email{wild@anl.gov}).}

\renewcommand{\thefootnote}{\arabic{footnote}}
\begin{abstract}
  Local Fourier analysis is a useful  tool for predicting  and analyzing the performance of many efficient algorithms for the solution of discretized PDEs, such as multigrid and domain decomposition methods. The crucial aspect  of local Fourier analysis is  that it can be used to minimize an estimate of the  spectral radius of a stationary iteration, or the condition number of a preconditioned system, in terms of a symbol representation of the algorithm.
  In practice, this is a ``minimax'' problem, minimizing with respect to solver parameters the appropriate measure of solver work, which involves maximizing over the Fourier frequency.
  Often, several algorithmic parameters may be determined by local Fourier analysis in order to obtain efficient algorithms.
  Analytical solutions to minimax problems are rarely possible beyond simple problems; the status quo in local Fourier analysis involves grid sampling, which is prohibitively expensive
in high dimensions.
  In this paper, we propose and explore optimization algorithms to solve these problems efficiently.
  Several examples, with known and unknown analytical solutions, are presented to show the effectiveness of these approaches.

\end{abstract}



%
%
%
%
%

\section{Introduction}\label{Intr-Optimization}
Multigrid methods \cite{MR1156079,MR558216,stuben1982multigrid,hackbusch2013multi,briggs2000multigrid}  have been successfully developed for the numerical solution of many discretized partial differential equations (PDEs), leading to broadly applicable algorithms that solve problems with $N$ unknowns in $O(N)$ work and storage. Constructing  efficient multigrid methods depends heavily  on the choice of the algorithmic components, such as the coarse-grid correction,
prolongation, restriction, and relaxation schemes.  No general rules for such choices exist, however, and many problem-dependent decisions must be made. Local Fourier analysis (LFA) \cite{MR1807961,wienands2004practical} is a commonly used tool for analyzing multigrid and other multilevel algorithms. LFA was first introduced by Brandt \cite{brandt1977multi}, where the LFA smoothing factor was presented as a good predictor for multigrid performance. The principal
advantage of LFA is that it provides relatively sharp quantitative estimates of the asymptotic
multigrid convergence factor for some classes of relaxation schemes
and multigrid algorithms applied to linear PDEs possessing appropriate \revise{(generalizations of Toeplitz)} structure.

LFA has been applied to many PDEs, deriving optimal
algorithmic parameters in a variety of settings.  LFA for
pointwise relaxation of scalar PDEs is covered in standard textbooks
\cite{MR1807961, briggs2000multigrid} and substantial work has been done
to develop and apply LFA to coupled systems.  For the Stokes
equations, LFA was first presented for distributive
relaxation applied to the staggered marker-and-cell (MAC)
finite-difference discretization scheme
in \cite{niestegge1990analysis} and later for multiplicative
Vanka relaxation for both the MAC scheme and the Taylor-Hood
finite-element discretization \cite{sivaloganathan1991use,
  SPMacLachlan_CWOosterlee_2011a}.  Recently, LFA has been used to
analytically optimize relaxation parameters for additive variants of
standard block-structured relaxations for both finite-difference
and finite-element discretizations of the Stokes equations
\cite{NLA2147, HMFEMStokes}.  These works, in particular, show the
importance of choosing proper relaxation parameters to improve
multigrid performance.  Similar work has been done for several other
common coupled systems, such as poroelasticity
\cite{franco2018multigrid, luo2017uzawa, luo2018monolithic} or
Stokes-Darcy flow \cite{Luo_etal_2017a}.  With insights gained from
this and similar work, LFA has also been applied in much broader
settings, such as the optimization of algorithmic parameters in
balancing domain decomposition by constraints
preconditioners \cite{JedBDDC}.  While the use of a Fourier
ansatz inherently limits LFA to a class of homogeneous (or periodic \cite{bolten2018fourier,kumar2019local}) discretized
operators (block-structured with (multilevel) Toeplitz blocks), recent work has shown
that LFA can be applied to increasingly complex
and challenging classes of both problems and solution algorithms,
limited only by one's ability to optimize parameters within the LFA
symbols.

This problem can often be cast as the minimization of the spectral
radius or norm of the Fourier representation of an error-propagation operator
with respect to its algorithmic parameters.
Because the computation of this spectral radius involves a
maximization over the Fourier frequency,
this problem can be formulated as one of minimax optimization.
For  problems with simple structure or few parameters, the analytical
solution of this minimax problem can be determined. In many cases, however,
the solution of this minimax optimization problem is difficult
because of its nonconvexity and nonsmoothness.

In previous work, optimization approaches have been used to
tune performance of multigrid methods.
In \cite{CWOosterlee_RWienands_2002a}, a genetic
algorithm is used to select the multigrid components for discretized PDEs from a discrete set such that an approximate three-grid Fourier convergence factor is minimized.
Relatedly, in \cite{schmitt2019optimizing} a two-step optimization procedure is presented;
in the first step a genetic algorithm is applied to a biobjective optimization formulation
(one objective is the approximate LFA convergence factor) to choose multigrid components from a discrete set, and in the second step a single-objective evolutionary algorithm (CMA-ES) is used to tune values of relaxation parameters from the approximate Pareto points computed in the first step.
However, neither of these approaches directly solves the minimax problem derived from LFA and hence neither achieve the same goal proposed here.
 Grebhahn et al. \cite{grebhahn2014optimizing} apply the domain-independent tool \texttt{SPL Conqueror} in a series of experiments to predict performance-optimal configurations of two geometric multigrid codes, but this tool is based largely on heuristics of estimating performance gains and does not attempt to optimize directly any particular objective.
 Recently, machine learning approaches have been used in learning multigrid parameters.
 Neural network training has been proposed to learn prolongation matrices for multigrid methods for 2D diffusion \revise{and related unstructured} problems to obtain fast numerical solvers \cite{Greenfeld2019Optimization,Luz2020LearningAMG}.
 By defining a stochastic convergence functional approximating the spectral radius of an iteration matrix,
 Katrutsa et al. \cite{Katrutsa2020} demonstrate good performance of stochastic gradient methods applied to the functional to learn parameters defining
 restriction and prolongation matrices.

Optimizing even a few parameters in a multigrid method using LFA can be a difficult task,  which can be done analytically for only a very small set of PDEs, discretizations, and solution algorithms; see, for example, \cite{NLA2147,HMFEMStokes}. In practical use,  we often sample the parameters on only a finite set to approximately optimize the LFA predictions. The accuracy of these predictions can be unsatisfactory since it may be computationally prohibitive to sample the parameter space finely enough to make good predictions. Thus, there is a need to design better optimization tools for LFA.  \revise{In this paper, we investigate the use of modern nonsmooth and derivative-free optimization algorithms to optimize algorithmic parameters within the context of LFA.  While tools from machine learning could also be applied to achieve this goal (either directly to the same LFA-based objective functions as are used here or to other proxies for the multigrid convergence factor, as in \cite{Greenfeld2019Optimization,Luz2020LearningAMG,Katrutsa2020}), the approaches proposed here are expected to more efficiently yield good parameter choices, under the restrictions of applicability that are inherent to LFA.}

After introducing the minimax problem for LFA in \Cref{LFA-background}, in \Cref{Optimization-Discussion} we discuss optimization algorithms for the LFA minimax problem.
We employ an approach based on outer approximations for minimax optimization problems \cite{menickelly2017derivative} to solve the minimax problem considered here.
This method dynamically samples the Fourier frequency space, which defines a sequence of relatively more tractable minimax optimization problems.
In our numerical experiments in \Cref{Numer}, we investigate the use of derivative-free and derivative-based variants of this method
 for several LFA optimization problems to validate our approach.
We conclude in \Cref{Conclusion} with additional remarks.


\section{Local Fourier Analysis}\label{LFA-background}
LFA \cite{wienands2004practical,MR1807961} is a tool for predicting the actual performance of multigrid and other multilevel algorithms. The fundamental idea behind LFA is to transform a given discretized PDE operator and the error-propagation operator of its solution algorithm  into  small-sized matrix representations known as Fourier symbols. To design efficient algorithms using LFA, we analyze and optimize the properties of the Fourier representation instead of directly optimizing properties of the discrete systems themselves.  This reduction in problem size is necessary in order to analytically optimize simple methods and to numerically optimize methods with many parameters.
Such use of LFA leads to a minimax problem, minimizing an appropriate measure of work in the solver over its parameters, measured by maximizing the predicted convergence factor over all Fourier frequencies. While the focus of this paper is on solving these minimax problems, we first introduce some basic definitions of LFA; for more details, see \cite{wienands2004practical,MR1807961}.

\subsection{Definitions and Notation}
We assume that the PDEs targeted here are posed on domains in $d$ dimensions. A standard approach to LFA is to consider infinite-mesh problems, thereby avoiding treatment of boundary conditions.  Thus, we consider $d$-dimensional uniform infinite  grids
\begin{equation*}
  \mathbf{G}_{h}=\big\{\boldsymbol{x}:=(x_1,x_2,\ldots,x_d)=\boldsymbol{k}h=(k_{1},k_{2},\ldots,k_d)h, \; k_i\in \mathbb{Z}\big\},
\end{equation*}
where we use the subscript $h$
to indicate an operator discretized with a meshsize $h$.

Let $L_h$ be a scalar Toeplitz or multilevel Toeplitz operator, representing a discretization of a scalar PDE, defined by its stencil acting on $l^\infty(\mathbf{G}_{h})$ as follows,
\begin{eqnarray*}\label{defi-symbol-P2P1}
  L_{h}  &:=& [s_{\boldsymbol{\kappa}}]_{h} \,\,(\boldsymbol{\kappa}=(\kappa_{1},\kappa_{2},\cdots,\kappa_{d})\in \boldsymbol{V}); \, \quad
  L_{h}w_{h}(\boldsymbol{x})=\sum_{\boldsymbol{\kappa}\in\boldsymbol{V}}s_{\boldsymbol{\kappa}}w_{h}(\boldsymbol{x}+\boldsymbol{\kappa}h),
  \end{eqnarray*}
with constant coefficients $s_{\boldsymbol{\kappa}}\in \R$ (or $\mathbb{C}$), $\boldsymbol{V}\subset \mathbb{Z}^d$ a finite index set, and $w_{h}(\boldsymbol{x})$
a bounded function on $\mathbf{G}_{h}$. Because $L_h$ is formally diagonalized by the Fourier modes $\varphi(\thetab,\boldsymbol{x})= e^{\iota\thetab\cdot\boldsymbol{x}/h}$, where $\thetab=(\theta_1,\theta_2,\ldots,\theta_d)$ and $\iota^2=-1$, we use $\varphi(\thetab,\boldsymbol{x})$ as a Fourier basis with $\thetab\in \big[-\frac{\pi}{2},\frac{3\pi}{2}\big)^{d}$ (or any  product of intervals with length $2\pi$).

\begin{definition}\label{formulation-symbol}
  We call $\widetilde{L}_{h}(\thetab)=\sum_{\boldsymbol{\kappa}\in\boldsymbol{V}}s_{\boldsymbol{\kappa}}e^{\iota \thetab\cdot\boldsymbol{\kappa}}$ the symbol of $L_{h}$.
\end{definition}
 Note that for all  functions $\varphi(\thetab,\boldsymbol{x})$, $L_{h}\varphi(\thetab,\boldsymbol{x})= \widetilde{L}_{h} (\thetab)\varphi(\thetab,\boldsymbol{x}).$

Here, we focus on the design and use of  multigrid methods for solving  such discretized PDEs. The choice of multigrid components is critical to attaining rapid convergence. In general, multigrid methods make use of complementary relaxation and coarse-grid correction processes to build effective iterative solvers. While  algebraic multigrid \cite{ruge1987algebraic} is effective in some settings where geometric multigrid is not, our focus is on the design and optimization of relaxation schemes for problems posed on regular meshes, to complement a geometric coarse-grid correction process.  The error-propagation operator for a relaxation scheme, represented similarly by a (multilevel) Toeplitz operator,  $M_h$, applied to $L_h$ is generally written as
\begin{equation*}
\mathcal{S}_h(\pb)=I-M_h^{-1}{L}_h,
\end{equation*}
where $\pb\in \R^n$
\revise{can represent classical relaxation weights or other parameters within the relaxation scheme represented by $M_h$.}
For geometric multigrid to be effective, the relaxation scheme $\mathcal{S}_h$ should reduce high-frequency error components quickly but can be slow to reduce low-frequency errors.
 Here, we consider standard geometric grid coarsening; that is, we construct a sequence of coarse grids by doubling the mesh size in each direction. The coarse grid, $\mathbf{G}_{H}$, is defined similarly to $\mathbf{G}_{h}$.  Low  and high frequencies for standard coarsening (as considered here) are given by
\begin{equation*}\label{Low-Freq}
 T^{{\rm low}} :=\left[-\frac{\pi}{2},\frac{\pi}{2}\right)^{d} \quad \mbox{and }  \quad  T^{{\rm high}} :=\left[-\frac{\pi}{2},\frac{3\pi}{2}\right)^{d}\Big\backslash T^{\rm low}.
\end{equation*}

It is thus natural to define the following LFA smoothing factor.

\begin{definition}\label{def-S-mu}
 The error-propagation symbol, $\widetilde{\mathcal{S}}_{h}(\pb,\thetab)$, for relaxation scheme $\mathcal{S}_{h}(\pb)$ on the infinite grid  $\mathbf{G}_{h}$ satisfies
\begin{equation*}
  \mathcal{S}_{h}(\pb)\varphi(\thetab,\boldsymbol{ x})=\widetilde{\mathcal{S}}_{h}(\pb,\thetab)\varphi(\thetab,\boldsymbol{ x}), \,\; \thetab\in \Big[-\frac{\pi}{2},\frac{3\pi}{2}\Big)^d,
\end{equation*}
for all $\varphi(\thetab,\boldsymbol{ x})$, and the corresponding smoothing factor for $\mathcal{S}_{h}$ is given by
\begin{equation}\label{LFA-mu}
  \mu=\max_{\thetab\in T^{{
  \rm high}}}\Big\{\big|\widetilde{\mathcal{S}}_{h}(\pb,\thetab)\big|\Big\}.
\end{equation}
\end{definition}

If the smoothing factor $\mu$ is small, then the dominant error after relaxation is associated with the low-frequency modes on $\mathbf{G}_h$. These modes can be well approximated by their analogues on the coarse grid $\mathbf{G}_H$, which is inherently cheaper to deal with in comparison to $\mathbf{G}_h$.  This leads to a two-grid method---see \cite{MR1156079,MR558216}---where relaxation on $\mathbf{G}_h$ is complemented with a direct solve for \revise{an approximation of the error on $\mathbf{G}_H$}.  \revise{To fix notation, we introduce a standard two-grid algorithm in \Cref{TG-Alg-process}.  Here, the relaxation process is specified by $M_h$, as well as the number of pre- and post-relaxation steps, $\nu_1$ and $\nu_2$.  The coarse-grid correction process is specified by the coarse-grid operator, $L_H$, as well as the interpolation and restriction operators, $P_h$ and $R_h$, respectively.  The error-propagation operator for this algorithm is given by}
\begin{equation}\label{Two-grid-matrix}
  E = \big(I-M_h^{-1}L_h\big)^{\nu_2}\big(I - P_hL_H^{-1}R_hL_h\big)\big(I-M_h^{-1}L_h\big)^{\nu_1}.
\end{equation}
Solving the coarse-grid problem recursively by the two-grid method yields a multigrid method, and  \cref{Two-grid-matrix} can be extended to the corresponding multigrid error-propagation operator.  Varieties of multigrid methods have been developed and used, including W-, V-, and F-cycles; here, we focus on two-level methods.
 \begin{algorithm}\caption{Two-grid method: $u_h^{j+1}={\bf TG}(L_h,M_h,L_H, P_h, R_h, b_h,u_h^{j},\nu_1,\nu_2)$}\label{TG-Alg-process}
 \begin{enumerate}
 \item Pre-relaxation:  Apply $\nu_1$ sweeps of relaxation to $u_h^{j}$:
 \begin{equation*}\label{Pre-SM}
\tilde{u}_h^{j}= {\bf Relaxation}^{\nu_1}(L_h,M_h,b_h,u_h^{j}).
\end{equation*}
 \item Coarse grid correction (CGC):
 \begin{itemize}
 \item Compute the residual: $r_h = b_h-L_h\tilde{u}^{j}_h$;
 \item Restrict the residual: $r_H = R_h r_h$;
 \item Solve the coarse-grid problem: $L_H u_H =r_H$;
 \item Interpolate the correction: $\delta u_h = P_h u_H$;
 \item Update the corrected approximation:  $\hat{u}_h^{j} = \tilde{u}_h^{j} + \delta u_h$;
 \end{itemize}
 \item Post-relaxation: Apply $\nu_2$ sweeps of relaxation to $\hat{u}_h^{j},$
 \begin{equation*}\label{Post-SM}
 u_h^{j+1}= {\bf Relaxation}^{\nu_2}(L_h,M_h,b_h,\hat{u}_h^{j}).
 \end{equation*}
 \end{enumerate}
 \end{algorithm}

For many real-world problems, the operator $L_h$ can be much more complicated than a scalar PDE. To extend LFA to coupled systems of PDEs or higher-order discretizations, one often needs to consider  $q\times q$ linear systems of operators,
\begin{equation*}
  L_h=\begin{pmatrix}
       L_h^{1,1}&   \cdots & L_h^{1,q}  \\
      \vdots&   \cdots & \vdots \\
       L_h^{q,1}&   \cdots & L_h^{q,q} \\
    \end{pmatrix}.
\end{equation*}
Here, $L_h^{i,j}$
denotes a  scalar (multilevel) Toeplitz operator describing how component $j$ in the solution appears in the equations of component $i$, where the number of components, $q$, is determined by the PDE itself and its discretization; see, for example, \cite{StokePatchFHM,NLA2147,HMFEMStokes,HM2018LFALaplace}.
Each entry in the symbol $\widetilde{L}_{h}$ is computed, following  \cref{formulation-symbol}, as the (scalar) symbol of the corresponding block of $L^{i,j}_{h}$. For a given relaxation scheme, the symbol of the block operator $M_h$ can be found in the same way, leading to the extension of the smoothing factor from \cref{def-S-mu}, where the spectral radius of a matrix, $\rho\big(\widetilde{S}_h(\pb,\thetab)\big)$, appears in \cref{LFA-mu} in place of the scalar absolute value. \revise{As in the scalar case, $\rho\big(\widetilde{S}_h(\pb,\thetab)\big)$ represents the worst-case convergence estimate for asymptotic reduction of error in the space of modes at frequency $\thetab$ per step of relaxation.}

The spectral radius or norm  of  $E$ provides measures of the efficiency of the two-grid method. In practice, however, directly analyzing $E$ is difficult because of the large size of the discretized system and the (many) algorithmic parameters involved in relaxation. As a proxy for directly optimizing $\rho(E)$, LFA is often used to choose relaxation parameters to minimize either the smoothing factor $\mu$ or the  ``two-grid LFA convergence factor''
\begin{equation}\label{inner-function}
\Psi_{\Tlow}(\pb) := \displaystyle\max_{\thetab\in\Tlow}\Big\{\rho\big(\widetilde{ E}(\pb,\thetab)\big)\Big\},
\end{equation}
where $\widetilde{E}$ is the symbol of $E$ computed by using extensions of the above approach.  \revise{The parameter set, $\pb$, may now include relaxation parameters, as above, and/or parameters related to the coarse-grid correction process, such as under/over-relaxation weights or entries in the stencils for $P_h$, $R_h$, or $L_H$.}
To estimate the convergence factor of the two-grid operator $E$ in \cref{Two-grid-matrix} using LFA, one thus needs to compute the symbol $\widetilde E$ by examining how the operators $L_h, P_h, L_H, \mathcal{S}_h,$ and so on act on the Fourier components $\varphi(\thetab,\boldsymbol{x})$.

 Note that for any $\boldsymbol{\theta'} \in \left[-\frac{\pi}{2},\frac{\pi}{2}\right)^d$,
\begin{equation}\label{distiguish-LH}
 \varphi(\thetab,\boldsymbol{x}) =\varphi(\boldsymbol{\theta'},\boldsymbol {x})\,\, \text{for}\, \boldsymbol{x} \in \mathbf{G}_H,
\end{equation}
if and only if $\thetab=\boldsymbol{\theta'}({\rm mod}\,\, \pi)$. This means that only those frequency components $\varphi(\thetab,\cdot)$  with $\boldsymbol {\theta}\in \left[-\frac{\pi}{2},\frac{\pi}{2}\right)^d$ are distinguishable on $\mathbf{G}_{H}$.
From \cref{distiguish-LH}, we know that there are $(2^d-1)$ harmonic frequencies, $\thetab$,  of $\boldsymbol{\theta'}$ such that $\varphi(\thetab,\boldsymbol{x})$ coincides with $\varphi(\boldsymbol{\theta'},\boldsymbol {x})$ on $\mathbf{G}_{H}$.
Let
\begin{eqnarray*}
I_{\boldsymbol{\alpha}}:&=&\big\{\boldsymbol{\alpha}=(\alpha_1,\alpha_2,\cdots,\alpha_d) : \;  \alpha_j\in\{0,1\}, \, j=1,\ldots,d\big\},\\
\thetab^{\boldsymbol{\alpha}}&=&(\theta_1^{\alpha_1},\theta_2^{\alpha_2},\cdots,\theta_d^{\alpha_d})=\thetab+\pi\cdot\boldsymbol{\alpha},\qquad
\thetab\in T^{{\rm low}}.
\end{eqnarray*}
Given $\thetab\in T^{\rm low}$, we define a $2^d$-dimensional space
$
  \mathcal{F}(\thetab)={\rm span}\left\{ \varphi(\boldsymbol{\theta^{\alpha}},\cdot): \, \boldsymbol{\alpha}\in I_{\boldsymbol{\alpha}}\right\}.
$
Under reasonable assumptions, the space $\mathcal{F}(\thetab)$ is invariant under the two-grid operator $E$ \cite{MR1807961}.

Inserting the representations of $\mathcal{S}_h, L_h, L_{H}, P_h$, and  $R_h$ into \cref{Two-grid-matrix}, we obtain the Fourier representation of the two-grid error-propagation operator \cite[{\rm Section} 4]{MR1807961} as
\begin{equation}\label{Two-grid-symbol}
 \widetilde{E}(\pb,\thetab)= \widetilde{\boldsymbol {S}}^{\nu_2}_h(\pb,\thetab)\Big[I-\widetilde{\boldsymbol {P}}_h(\thetab)\big(\widetilde{L}_{H}(\pb,2\thetab)\big)^{-1}\widetilde{\boldsymbol{ R}}_h(\thetab)\widetilde{\boldsymbol{ \mathcal{L}}}_{h}(\thetab)\Big]\widetilde{\boldsymbol{ S}}^{\nu_1}_h(\pb,\thetab),
\end{equation}
where
\begin{eqnarray*}
\widetilde{\boldsymbol{\mathcal{L}}}_h(\thetab)&=&\text{diag}\left\{\widetilde{L}_h(\thetab^{\boldsymbol{\alpha}_1}),\cdots, \widetilde{L}_h(\thetab^{\boldsymbol{\alpha}_j}),\cdots,\widetilde{L}_h(\thetab^{\boldsymbol{\alpha}_{2^d}})\right\},\\
\widetilde{\boldsymbol{S}}_h(\pb, \thetab)&=&\text{diag}\left\{\widetilde{\mathcal{S}}_h(\pb, \thetab^{\boldsymbol{\alpha}_1}),
\cdots,\widetilde{\mathcal{S}}_h(\pb, \thetab^{\boldsymbol{\alpha}_j}),\cdots,
\widetilde{\mathcal{S}}_h(\pb,\thetab^{\boldsymbol{\alpha}_{2^d}})\right\},\\
\widetilde{\boldsymbol{R}}_h(\thetab)&=&\left(\widetilde{R}_h(\thetab^{\boldsymbol{\alpha}_1}),\hspace{2mm}\cdots,
\widetilde{R}_h(\thetab^{\boldsymbol{\alpha}_j}),\hspace{2mm}\cdots,
\widetilde{R}_h(\thetab^{\boldsymbol{\alpha}_{2^d}}) \right),\\
\widetilde{\boldsymbol{P}}_h(\thetab)&=&\left(\widetilde{P}_h(\thetab^{\boldsymbol{\alpha}_1}); \hspace{2mm} \cdots;
\widetilde{P}_h(\thetab^{\boldsymbol{\alpha}_j});\hspace{2mm} \cdots;\widetilde{P}_h(\thetab^{\boldsymbol{\alpha}_{2^d}}) \right),
\end{eqnarray*}
in which ${\rm diag}\{T_1,\cdots,T_{2^d}\}$ stands for the block diagonal matrix with diagonal blocks $T_1$ through $T_{2^d}$, block-row and block-column matrices are represented by $(T_1, \cdots, T_{2^d})$ and $(T_1; \cdots; T_{2^d})$, respectively, and $\{\boldsymbol{\alpha}_1,\cdots,
\boldsymbol{\alpha}_{2^d}\}$ is
the set  $I_{\boldsymbol{\alpha}}$.  \revise{We note that $\widetilde{L}_{H}$ is sampled at frequency $2\thetab$ in \cref{Two-grid-symbol} because we use factor-two coarsening.  While $P_h$ and $R_h$ may also depend on $\pb$, we suppress this notation and only consider $L_H$ and $\mathcal{S}_h$ to be dependent on $\pb$, as is used in the examples below.}

In many cases, the convergence factor of the two-grid method in \cref{Two-grid-matrix} can be estimated directly from
the LFA smoothing factor in \cref{def-S-mu}.  In particular, if we have an ``ideal''
coarse-grid correction operator that annihilates low-frequency error components and
leaves high-frequency components unchanged, then the resulting LFA smoothing
factor usually gives a good prediction for the actual multigrid performance.  In this case, we need to optimize only the smoothing factor, which is simpler than optimizing the two-grid LFA convergence factor. For some discretizations, however, $\mu$ fails to accurately predict the actual two-grid performance; see, for example, \cite{HM2018LFALaplace,HMFEMStokes,SPMacLachlan_CWOosterlee_2011a}.
Thus, we focus on the two-grid LFA convergence factor, which accounts for coupling between different harmonic frequencies and often gives a sharp prediction of the actual two-grid performance. In  \Cref{Numer}, we present several examples of optimizing two-grid LFA convergence factors.

\begin{remark}
While the periodicity of Fourier representations naturally leads to minimax problems over half-open intervals such as $T^{\rm low}$, computational optimization is more naturally handled over closed sets.  Thus, in what follows, we optimize the two-grid convergence factor over the closure of the given interval.
\end{remark}

\subsection{Minimax Problem in LFA}
The main goal in our use of LFA is to find an approximate solution
of the minimax problem
\begin{equation}\label{opt-problem}
\min_{\pb\in \cD} \Psi_{\Tlow}(\pb) = \min_{\pb\in \cD}\max_{\thetab\in T^{{\rm low}}}\big\{\rho\big(\widetilde{ E}(\pb,\thetab)\big)\big\},
\end{equation}
where $\widetilde{E}$ is  as discussed above and $\pb$ are algorithmic parameters in $\cD$,  a compact set of allowable parameters. In general, we think of $\cD$ as being implicitly defined as
\begin{equation*}
  \cD=\left\{\pb: \rho\big(\widetilde{ E}(\pb,\thetab)\big)\leq 1 \; \; \forall \thetab\in T^{\rm low}\right\};
\end{equation*}
in practice, we typically use an explicit definition of $\cD$ that is expected to contain this minimal set (and, most important, the optimal $\pb$).
Generalizations of \cref{opt-problem} are also possible, such as to minimize the condition number of a preconditioned system corresponding to other solution approaches.
For example, the authors of \cite{JedBDDC} use LFA to study the condition numbers of the preconditioned system in one type of domain decomposition method, where a different minimax problem arises.
Similarly, in \cite{de2019optimizing}, the optimization of $\|\widetilde{E}(\pb,\thetab)\|$ is considered in the highly non-normal case.

The minimax problem  in \cref{opt-problem} is often difficult to solve.  First, the minimization over $\pb$ in \cref{opt-problem} is generally a nonconvex optimization problem. Second, the minimization over $\pb$ in \cref{opt-problem} is generally a \emph{nonsmooth} optimization problem.
There are two sources of nonsmoothness in \cref{opt-problem}; most apparently, the maximum over $\thetab\in T^{\rm{low}}$ creates a nonsmooth minimization problem over $\pb$, which we will discuss further in the next section. Perhaps more subtle, for fixed $(\pb,\thetab)$, the eigenvalue that realizes the maximum in $\rho(\cdot)$ need not be unique;
in this case, a gradient may not exist, and so $\rho(\widetilde{E}(\pb,\thetab))$ is generally not a differentiable function of $\pb$.
Third, the matrix $\widetilde{E}$ is often not Hermitian,
and there is no specific structure that can be taken advantage of. Thus, while analytical solution of the minimax problem is sometimes possible, one commonly only approximates the solution. Often, this approximation is accomplished simply by a brute-force search over discrete sets of Fourier frequencies and algorithmic parameters, as described next.

\subsection{Brute-Force Discretized Search}\label{subsec:BFS}

Since $T^{\rm low}$ is of infinite cardinality, \cref{opt-problem} is a semi-infinite optimization problem.
\revise{The} current practice is to consider a discrete form of \cref{opt-problem}, resulting from  brute-force sampling over only a finite set of points
$(\pb,\thetab)$.
Assuming $\cD$ is an $n$-dimensional rectangle, this can mean sampling $N_p$ points  in each dimension of the parameter $\pb$ on the intervals $a_k\leq p_k\leq b_k$ and $\Ntheta$
points in each dimension of the frequency $\thetab$.

A simple brute-force discretized search algorithm is given in \Cref{alg:brute-force}.
For convenience, we sample points evenly in each interval, but other choices can also be used.  Since there are $N_p^n$ sampling points in parameter space and $\Ntheta^d$ sampling points in frequency space, $N_p^n \cdot \Ntheta^d$ spectral radii
%
need to be computed within the algorithm.  With increasing values of $q$, $n$, and $d$,  this requires substantial computational work that is infeasible for large values of $d$ and $n$.\footnote{For example, assume we take $N_p=20$ and $\Ntheta=33$ sample points in parameters and frequency, respectively.  Then, the total number of evaluations of $\rho\big(\widetilde{E}(\pb,\thetab)\big)$ is $W(n,d) = 20^{n}\cdot 33^{d}$.
For 2D and 3D PDEs with 3 and 5 algorithmic parameters, this would yield $W(3,2) \approx 9\cdot 10^{6}$,
$W(5,2) \approx 3 \cdot 10^{9}$,
$W(3,3) \approx 3 \cdot 10^{8}$, and
$W(5,3) \approx 10^{11}$.}
Thus, there is a clear need for efficient optimization algorithms for LFA. The aim of this paper is to solve \cref{opt-problem} efficiently without loss of accuracy in the solution.

\section{Robust Optimization Methods}\label{Optimization-Discussion}


The minimax problem \cref{opt-problem}
involves minimization of a function $\Psi_{\Tlow}$ that is generally nonsmooth and defined by the ``inner problem'' of \cref{inner-function}.
%
%
When $\Tlow$ has infinite cardinality (e.g., a connected range of frequencies), solving the maximization
problem in \cref{inner-function} is generally intractable; it results in a global optimization problem
provided $\rho\big(\widetilde{ E}(\pb,\thetab)\big)$ is nonconvex in $\thetab$.
One attempt to mitigate this difficulty, which we test in our numerical results, is to relax the inner problem
\cref{inner-function} by replacing $\Tlow$ with
a finite discretization.

\subsection{Fixed Inner Discretization}\label{HANSO}
Given $\Ntheta^d$ frequencies
$$\thetab(\Ntheta^d):=\{\thetab_j: j=1,\dots,\Ntheta^d\}\subset \Tlow,$$
we can define a relaxed version of the inner problem \cref{inner-function},
\begin{equation}\label{inner-function-2}
\Psi_{\thetab(\Ntheta^d)}(\pb) := \displaystyle\max_{\thetab_j\in\thetab(\Ntheta^d)}\Big\{\rho\big(\widetilde{ E}(\pb,\thetab_j)\big)\Big\}
\;
\leq
\;
\Psi_{\Tlow}(\pb).
\end{equation}
A single evaluation of $\Psi_{\thetab(\Ntheta^d)}(\pb)$ can be performed in finite time by performing $\Ntheta^d$ evaluations of
$\rho\big(\widetilde{ E}(\pb,\cdot)\big)$ and recording the maximum value. This is the ``inner loop'' performed over $\Ntheta^d$ points in lines 9--19 of  \Cref{alg:brute-force}.

Provided there exist gradients with respect to $\pb$ for a fixed value of $\thetab$,
\begin{equation}\label{gradients}
\nabla_{\pb} \rho\big(\widetilde{E}(\pb,\thetab)\big),
\end{equation}
 we automatically obtain a
 subgradient of
 $\Psi_{\thetab(\Ntheta^d)}(\pb)$
 by choosing an arbitrary index
 $$j^*\in\operatorname*{argmax}_{j=1,\dots,\Ntheta^d} \Big\{\rho\big(\widetilde{E}(\pb,\thetab_j)\big)\Big\}$$
 and returning $\nabla_{\pb} \rho\big(\widetilde{E}(\pb,\thetab_{j^*})\big)$.
It is thus straightforward to apply nonsmooth optimization algorithms
 that query a (sub)gradient at each point to the minimization of the relaxation $\Psi_{\thetab(\Ntheta^d)}(\pb)$ in \cref{inner-function-2}.
One such algorithm employing derivative information is \HANSO\ \cite{burke2005robust}; in preliminary experiments, we found \HANSO\ to be suitable when applied to the minimization of $\Psi_{\thetab(\Ntheta^d)}(\pb)$.

This sampling-based approach of applying methods for nonconvex nonsmooth optimization still requires the specification of
$\thetab(\Ntheta^d)$, which may be difficult to do a priori; it may amount to brute-force sampling
of the frequency space.

\subsection{Direct Minimax Solution}\label{ROBOBOA}
We additionally consider inexact methods of outer approximation \cite{Polak1997},
 in particular those for minimax optimization
 \cite{menickelly2017derivative},
 to directly solve \cref{opt-problem}.
 The defining feature of such methods is that, rather than minimizing a fixed relaxation $\Psi_{\thetab(\Ntheta^d)}(\pb)$
in the $k$th outer iteration, they solve
 \begin{equation}\label{opt-problem-3}
\displaystyle\min_{\pb\in\cD} \left\{ \Psi_{\cU_k}(\pb):= \displaystyle\max_{\thetab_j\in\cU_k}\big\{\rho\big(\widetilde{E}(\pb,\thetab_j)\big)\big\} \right\},
 \end{equation}
wherein the relaxation is based on a set $\cU_k\subset \Tlow$ that is defined iteratively.
 In the $k$th iteration of the procedure used here, a nonsmooth optimization method known as manifold sampling \cite{LMW16, KLW18}
 is applied to obtain an approximate solution $\pb_k$ to \cref{opt-problem-3}.
 A new frequency $\thetab_k$ is then computed as an approximate maximizer of
 \begin{equation}\label{max-problem}
 \displaystyle\max_{\thetab\in\Tlow} \Big\{\rho\big(\widetilde{E}(\pb_k,\thetab)\big)\Big\},
 \end{equation}
and the set  $\cU_k$ is augmented so that $\thetab_k\in \cU_{k+1}$.
This two-phase optimization algorithm, which alternately finds approximate solutions to
\cref{opt-problem-3} and \cref{max-problem},
is guaranteed to find Clarke-stationary points
(i.e., points where 0 is a subgradient of $\Psi_{\Tlow}(\pb)$ when $\pb$ is on the interior of $\cD$)
of \cref{opt-problem}
as $k\to\infty$, given basic assumptions \cite{menickelly2017derivative}.

\begin{remark}
The work of \cite{menickelly2017derivative}
 (and, in particular, the manifold sampling algorithm employed to approximately solve \cref{opt-problem-3})
was originally intended for problems in which the (sub)gradients
$\partial_{\pb} \rho\left(\widetilde{E}(\pb,\thetab)\right)$ and
$\partial_{\thetab} \rho\left(\widetilde{E}(\pb,\thetab)\right)$
are assumed to be unavailable.
 However, the method in \cite{menickelly2017derivative}
can exploit such derivative information when it is available. Using the concept of  model-based methods for derivative-free optimization
(see, e.g., \cite[Chapter 10]{Conn2009a}, \cite[Section 2.2]{LMW2019AN}), manifold sampling depends on the construction at $\pb_k$ of local
models of functions $\rho\big(\widetilde{E}(\cdot,\thetab_j)\big)$ for a subset of $\thetab_j\in\cU_k$,
each of which is as accurate as a first-order Taylor model of $\rho\big(\widetilde{E}(\cdot,\thetab_j)\big)$.
\revise{When gradient information is available, a gradient can be employed to construct the exact first-order Taylor model of
$\rho\big(\widetilde{E}(\cdot,\thetab_j)\big)$ around $\pb_k$.}
In our numerical results, we test this method both with and without derivative information.
\end{remark}

\subsection{Computation of Derivatives}
In our setting,  computing (sub)gradients as in \cref{gradients} can be nontrivial.
Here we discuss possible approaches.
Note that for any point $(\pb,\thetab)$, $\rho\big(\widetilde{E}(\pb,\thetab)\big)$ is the absolute value of some eigenvalue of $\widetilde{E}$. Thus, of interest are the following derivatives of eigenpairs of a matrix $\widetilde{E}$.

\begin{theorem}\label{Thm:derivative}
  Suppose that, for all $\pb$ in an open neighborhood of $\hat{\pb}$, $\lambda(\pb)$ is a simple eigenvalue of $\widetilde{E}(\pb)$ with right eigenvector $\xb(\pb)$ and left eigenvector $\yb(\pb)$ such that $\yb^T\xb(\pb) \neq 0$ and that $\widetilde{E}(\pb)$, $\lambda(\pb)$, and either $\xb(\pb)$ or $\yb(\pb)$ are differentiable at $\pb = \hat{\pb}$.
Then, at $\pb=\hat{\pb}$,
\begin{equation*}
          \frac{d\lambda}{dp_j}=\frac{\yb^T\frac{d\widetilde{E}}{dp_j}\xb}{\yb^T\xb} \,\,{\rm and}\,\,\,\,
          \frac{d|\lambda|}{dp_j}={\rm Re}\Big(\frac{\overline{\lambda}}{|\lambda|} \frac{d\lambda}{dp_j}\Big), \,\,{\rm for} \,\,j=1,2,\cdots,n,
\end{equation*}
where $\overline{\lambda}$ is the complex conjugate of $\lambda$, ${\rm Re}(z)$ is the real part of $z\in \mathbb{C}$, and the second expression only holds if $\lambda(\hat{\pb})\neq 0$.
\end{theorem}
\begin{proof}
Differentiating both sides of $\widetilde{E} \xb=\lambda \xb$, we have
\begin{equation}\label{derivativeA}
  \frac{d(\widetilde{E}\xb)}{d p_j} =\frac{d\widetilde{E}}{dp_j}\xb+ \widetilde{E}\frac{d\xb}{dp_j}=\frac{d \lambda}{dp_j}\xb+\lambda\frac{d\xb}{dp_j}.
\end{equation}
Multiplying by $\yb^T$ on the left side of \cref{derivativeA}, we obtain
\begin{equation*}
  \yb^T\frac{d\widetilde{E}}{dp_j}\xb+ \yb^T\widetilde{E}\frac{d\xb}{dp_j}=\yb^T\frac{d \lambda}{dp_j}\xb+\yb^T\lambda\frac{d\xb}{dp_j}.
\end{equation*}
Noting that $\yb^T\widetilde{E}=\lambda \yb^T$ and since $\yb^T\xb\neq 0$,
we arrive at the first statement,
$\frac{d\lambda}{dp_j}=\frac{\yb^T\frac{d\widetilde{E}}{dp_j}\xb}{\yb^T\xb}.$
The same conclusion holds if, instead, we differentiate $\yb^T\widetilde{E} = \lambda\yb^T$.
For $\lambda\neq 0$, we rewrite $|\lambda|=\sqrt{\lambda\overline{\lambda}}$ and see that
 \begin{eqnarray*}
   \frac{d|\lambda|}{dp_j} \, = \, \frac{d}{dp_j}(\lambda\overline{\lambda})^{1/2}
   &=& \frac{1}{2|\lambda|}\Big(\frac{d\lambda}{dp_j}\overline{\lambda} +\lambda\frac{d\overline{\lambda}}{dp_j}\Big)\\
    &=& \frac{1}{2|\lambda|}2 {\rm Re}\Big(\frac{d\lambda}{dp_j}\overline{\lambda}\Big) \,
     = \, {\rm Re} \Big(\frac{\overline{\lambda}}{|\lambda|}\frac{d\lambda}{d p_j}\Big),
 \end{eqnarray*}
which is the second statement.
\end{proof}

\begin{remark}
  Where $\widetilde{E}$ is diagonalizable, one can easily see that for any right eigenvector, $\xb$, there always exists a left eigenvector, $\yb$, such that $\yb^T\xb \neq 0$.  This follows from writing the eigenvector decomposition of $\widetilde{E} = V\Lambda V^{-1}$, so that the right eigenvectors are defined by $\widetilde{E}V = V\Lambda$ (i.e., any right eigenvector is a linear combination of columns of $V$) and the left eigenvectors are defined by $V^{-1}\widetilde{E} = \Lambda V^{-1}$ (i.e., any left eigenvector is a linear combination of rows of $V^{-1}$).  For any (scaled) column of $V$, there is a unique (scaled) row of $V^{-1}$ such that $\yb^T\xb \neq 0$ since $V^{-1}V = I$.  Thus, for eigenvalues of multiplicity 1, for any right eigenvector, $\xb$, there is a unique $\yb$ (up to scaling) such that $\yb^T\xb \neq 0$.  For eigenvalues of multiplicity greater than one, for any eigenvalue, $\lambda$, there exist multiple eigenvectors, $\xb$, taken as linear combinations of columns of $V$.  Taking the same linear combination of rows of $V^{-1}$ to define the left eigenvector, $\yb^T$, still satisfies the requirement.  In the non-diagonalizable case, this argument can be extended to simple or semi-simple eigenvalues within the Jordan normal form, $\widetilde{E} = VJV^{-1}$, but not directly to the degenerate case. While there is a long history of study of the question of differentiability of eigenvalues and eigenvectors (see, for example, \cite{magnus1985differentiating,MR1061136,lancaster1964eigenvalues}), we are unaware of results that cover the general case.
\end{remark}

In our numerical results, we employ the derivative expressions noted in \cref{Thm:derivative}. We note that the restrictive assumptions of \cref{Thm:derivative} may not always apply and, hence, the expressions at best correspond to subderivatives.  We emphasize that this is an important practical issue, as non-normal systems are common as components in multigrid methods, such as with (block) Gauss-Seidel or Uzawa schemes used as relaxation methods.  Even for simple examples, we observe frequency and parameter pairs where the symbol of a component of the multigrid method has either orthogonal left- and right-eigenvectors or non-differentiable eigenvectors.  Thus, there are more fundamental issues associated with the use of \cref{Thm:derivative} in addition to the usual numerical issues in computing eigenvalues of (non-normal) matrices.  Nonetheless, we successfully use the expressions above in this work.  A slightly simpler theoretical setting arises when considering $\|\widetilde{E}\|$ in place of $\rho(\widetilde{E})$ in \cref{max-problem}, similar to what was considered for the parallel-in-time case in \cite{de2019optimizing}.  The differences between this and the traditional approach of \cref{max-problem} are significant, though, so we defer these to future work.

In some settings, we can readily compute $\frac{d\widetilde{E}}{dp_j}$, providing the derivatives of the relevant eigenvalues as analytical expressions.
When derivative information is unavailable, one (sub)derivative approximation arises by considering central finite differences:
\begin{equation}\label{eq:subderivative}
  \frac{d\rho(\widetilde{E}(\pb,\thetab))}{dp_j} \approx\frac{\rho\big(\widetilde{ E}(\pb+t\boldsymbol{ e}_j,\thetab)\big)-\rho\big(\widetilde{ E}(\pb-t\boldsymbol{ e}_j,\thetab)\big)}{2t}, \qquad j=1, \ldots, n,
\end{equation}
where $\boldsymbol{e}_j$ is the $j$th canonical unit vector (acknowledging that this approximation can be problematic in cases of multiplicity or increasing $t>0$). We use $t=10^{-6}, 10^{-8}$, and $10^{-12}$ in our numerical experiments but note that automatic selection of $t$ \cite{more2011edn,pernice-walker-1998} and/or high-order finite-difference approaches can also be used.
Algorithmic differentiation
is also possible but presents additional challenges that we do not address here.
In the results below, we distinguish between optimization using analytical derivatives and central-difference approximations; furthermore, we use the term ``derivative'' to include the particular (sub)derivative obtained by analytical calculation or central-difference approximation.

\section{Numerical Results}\label{Numer}

We now study the above optimization approaches on 1D, 2D, and 3D problems obtained from the Fourier representation in \cref{Two-grid-symbol} of the two-grid error propagation operator \cref{Two-grid-matrix}.
Although we \revise{primarily} examine relaxation schemes (different $M_h$) in the two-grid method, other parameters can also be considered,  such as optimizing the coefficients in the grid-transfer operators.
All operators with a superscript tilde denote LFA symbols.  We emphasize that the subject of the optimization considered here is the convergence factor of a stationary iterative method, which should lie in the interval $[0,1]$ and is typically bounded away from both endpoints of the interval.  As such, optimization is needed only to one or two digits of accuracy, since the difference in performance between algorithms with convergence factors that differ by less than 0.01 is often not worth the expense of finding such a small improvement in the optimal values.  Likewise, the algorithms are generally not so sensitive to the parameter values beyond two digits of accuracy.

For our test problems, we consider both finite-difference and finite-element discretizations. We use two types of finite-element methods for  the Laplacian.  First, we consider the approximation for 1D meshes, using continuous
\revise{piecewise linear ($P_1$) approximations}. Second, we consider the
approximation for structured meshes of triangular elements using continuous \revise{piecewise quadratic ($P_2$) approximations}.  Similarly, we consider both finite-difference and finite-element discretizations of the Stokes problem, using the classical staggered (Marker and Cell, or MAC) finite-difference scheme, \revise{stabilized} equal-order bilinear ($Q_1$) finite-elements on quadrilateral meshes, and the stable Taylor-Hood ($P_2$-$P_1$) finite-element discretization on triangular meshes.  For the 3D optimal control problem, we consider the trilinear ($Q_1$) finite-element discretization on hexahedral meshes.
For details on these finite-element methods, see \cite{elman2006finite}.

\subsection{Optimization Methods}
We consider brute-force discretized search (\Cref{alg:brute-force}) as well as  different modes of two optimization approaches.

The first (``HANSO-FI'') is based on running the \HANSO\ code from \cite{LewOve12} to minimize the nonsmooth function in \cref{inner-function-2} resulting from a fixed inner sampling using $\Ntheta^d$ samples.
\revise{We chose \HANSO\ because it is capable of handling nonsmooth, nonconvex objectives. To the best of our knowledge, there are few such off-the-shelf solvers that have any form of guarantee concerning the minimization of \cref{inner-function-2}, and none that directly solve \cref{opt-problem-3}. We note that, given access to an oracle returning an (approximate) subgradient of \cref{inner-function-2}, solvers for smooth, nonconvex optimization could also be applied to the minimization of \cref{inner-function-2}; although such application can be problematic, 
it has been shown to work well in practice for some nonsmooth problems (see, e.g.,
 \cite{LewOve12}).
}

The second is based on running the \ROBOBOA\ code from \cite{menickelly2017derivative} to solve \cref{opt-problem-3} based on its adaptive sampling of the inner problem. In both cases we use either analytical (denoted by ``ROBOBOA-D, AG'') or central-difference approximate (denoted by ``ROBOBOA-D, FDG'' with the difference parameter $t$ indicated) gradients; for \ROBOBOA\ we also consider a derivative-free variant (``ROBOBOA-DF''), which does not directly employ analytical or approximate gradients.

For the test problems considered here, the parameters are damping parameters \revise{associated with either the relaxation scheme or the coarse-grid correction process};
thus we set $\cD=\Rp^d$.
Each method is initialized at the point corresponding to $p_i=0.5$ for $i=1,\ldots,n$, except for the $P_1$ discretization for the Laplacian with single relaxation and $P_2$ discretization for the Laplacian, where we use $p_i=0.1$.
\subsection{Performance Measures}
In all cases we measure performance in terms of the number of function (i.e., $\rho(\widetilde{E}(\cdot, \cdot))$) evaluations; in particular, we do not charge the cost associated with an analytical gradient, but we do charge the $2n+1$ function evaluations for the central-difference approximation.

\revise{Our first metric  for optimizing $\rho$} is based on a fixed, equally spaced sampling of $T^{\rm low}$:
\begin{equation}
\label{eq:rho_psi_metric}
\rho_{\Psi_*}(\hat{\pb}):=\Psi_{\thetab(\Ntheta^d=33^d)}(\hat{\pb}),
\end{equation}
where $\hat{\pb}$ is the best $\pb$ (with respect to this metric) obtained by the tested method.
When $\Ntheta$ is odd (as in \cref{eq:rho_psi_metric}), the sampling points include the frequency $\thetab=0$.
The symbol of the coarse-grid operator $L_H$ in \cref{Two-grid-matrix} is not invertible, since the vector with all ones is an eigenvector associated with eigenvalue zero. For our numerical tests, when $\thetab=0$, we approximate the limit
(which does exist) by setting the components of $\thetab$ to $10^{-7}$, \revise{a suitably small value, chosen experimentally to give a reliable approximation of the limit}.

\begin{remark}
In \cref{eq:rho_psi_metric}, we take $\Ntheta$ to be odd since, in most cases, the frequency near zero plays an important role in measuring and understanding the performance of the resulting multigrid methods.
\end{remark}

For the Laplace problems, we also show a second metric based on measured two-grid performance of the resulting multigrid methods, to validate the approximate parameters obtained by the optimization approaches. We consider the homogeneous problem, $A_h x_h =b= 0$, with discrete solution $x_{h} \equiv 0$, and start
  with a random initial $x_{h}^{(0)}$ to test the two-grid
  convergence factor. Rediscretization is used to define the coarse-grid operator, and
  we consider  Dirichlet boundary conditions for the $P_1$ discretization  and  periodic boundary conditions for the $P_2$ discretization. We
  use the defects \revise{after 100 two-grid cycles}, $d_{h}^{(100)}$ (with
  $d_h^{(100)}=b-A_hx_{h}^{(100)}$), to experimentally measure
performance of
the convergence factor
through the two estimates
\begin{align}
\begin{split}
  \rho_{m,1} :=&  \left(\|d_h^{(100)}\|_2/\|d_h^{(0)}\|_2\right)^{1/100},
\\
\rho_{m,2} :=& \|d_h^{(100)}\|_2/\|d_h^{(99)}\|_2;
\end{split}
\label{eq:defi-rho2}
\end{align}
see \cite{MR1807961}.
For the two-grid methods, the fine-grid mesh size is $h = \frac{1}{64}$.

We additionally consider an approximate measure of first-order stationarity of obtained solutions $\hat{\pb}$, denoted $\sigma(\hat{\pb})$;
details concerning $\sigma(\hat{\pb})$ are provided in \Cref{sec:params}.
Moreover, \Cref{sec:params} provides a summary of the approximate solutions $\hat{\pb}$
obtained by various solvers, as well as corresponding values of $\rho_{\Psi_*}(\hat{\pb})$ and $\sigma(\hat{\pb})$.

\revise{The examples considered below primarily take the form of validating our approach, by reproducing analytical solutions to the LFA optimization problems considered.  In some cases, these come from optimizing the LFA smoothing factor (as in \Cref{subsec:Q1-Q1-Stokes} and \Cref{sec:Control3D}), while others come from optimizing the LFA two-grid convergence factor.  In some cases (such as for the classical 1D Poisson operator with weighted-Jacobi relaxation, or for the various relaxation schemes for the MAC-scheme and Q1-Q1 finite-element discretizations of Stokes), corresponding analytical results are known in the literature.  For others (1D Poisson with 2 sweeps of weighted Jacobi and the 3D control problem), the results are new but can be verified analytically.  Finally, we also consider some cases where analytical results are unknown to us, namely for factor-three coarsening of 1D Poisson, the P2 discretization of Poisson in 2D, and the P2-P1 discretization of Stokes.}

\subsection{Poisson Equation}\label{sec:Poisson1D}
We first consider the
Poisson problem
\begin{equation*}
  \begin{aligned}
  -\Delta u(x) =f(x),\,\,\,\, x\in\Omega, \label{Laplace}\\
   u(x)=g(x),\,\,\, x\in \partial \Omega,\\
   \end{aligned}
\end{equation*}
using  the $P_1$ discretization in 1D and the $P_2$ discretization in 2D with simple Jacobi relaxation.
In 1D, there are two harmonics for each $\theta\in T^{\rm low}$, $\theta$ and $\theta+\pi$. For each harmonic, the LFA representation for the Laplace operator in 1D using the $P_1$ discretization is a scalar. We use standard stencil notation (see, for example, \cite{wienands2004practical}) to describe the discrete operators. The stencil for $-\Delta$ is
$
  L_h =\frac{1}{h}\begin{bmatrix} -1 & 2 & -1
                  \end{bmatrix}.
$
For  weighted Jacobi relaxation, we have
\begin{equation*}
  \mathcal{S}_h = I -p_1 M_h^{-1}L_h,
\end{equation*}
where the stencil for the diagonal matrix
$  M_h = \frac{1}{h}\begin{bmatrix}  2
                  \end{bmatrix}$.
We consider linear interpolation, with stencil representation
\begin{equation*}
P_h=\frac{1}{2}
\left ]
  \begin{tabular}{ccc}
  1 \\
  2  \\
  1
  \end{tabular}
\right [_{2h}^h,
\end{equation*}
and  $R_h =P_h^T$ for the restriction in the two-grid method. We denote $Q=I - P_hL_H^{-1}R_hL_h$, $S_1=I-p_{1} M_h^{-1}L_h$, and $S_2=I-p_{2} M_h^{-1}L_h$,
so that the error-propagation operator is $S_2^{\nu_2} Q S_1^{\nu_1}$ for a $TG(\nu_1,\nu_2)$-cycle (i.e., a two-grid $V(\nu_1,\nu_2)$-cycle).

We first consider $\nu_1=1,\nu_2=0$, and $E_1= QS_1$. According to \Cref{formulation-symbol} and \cref{Two-grid-symbol}, we have
\begin{equation*}
\widetilde{L}_{H}(2\theta) = \frac{1-\cos(2\theta)}{h},\,\quad
  \widetilde{\boldsymbol{\mathcal{L}}}_h = \frac{4}{h} \begin{pmatrix}
  s^2 & 0\\
  0 & c^2
  \end{pmatrix}, \,\quad \widetilde{\boldsymbol{\mathcal{M}}}_h = \frac{2}{h}\begin{pmatrix}
  1 & 0\\
  0 & 1
  \end{pmatrix},
\end{equation*}
where $s=\sin\frac{\theta}{2}$ and $c = \cos\frac{\theta}{2}$, and
\begin{equation*}
  \widetilde{ \boldsymbol{P}}_h = \frac{1}{2}\begin{pmatrix} 1+\cos(\theta)\\
  1-\cos(\theta)
  \end{pmatrix},\,\quad \widetilde{\boldsymbol{ R}}_h = 2\widetilde{\boldsymbol{P}}_h^T.
\end{equation*}
By a standard calculation, we have
\begin{equation*}
  \widetilde{\boldsymbol{ Q}} = \begin{pmatrix}s^2 & -c^2\\ -s^2 & c^2 \end{pmatrix}, \,\quad \widetilde{\boldsymbol{S}}_1 = \begin{pmatrix} 1-2 p_1 s^2 & 0 \\ 0 & 1-2p_1 c^2\end{pmatrix},
 \end{equation*}
which yields
\begin{equation}
\label{eq:Q1Lp}
          \widetilde{E}_1= \widetilde{\boldsymbol{ Q}}\widetilde{\boldsymbol{ S}}_1 = \begin{pmatrix} s^2(1-2p_1 s^2)  & -c^2(1-2p_1 c^2)\\
          -s^2(1-2p_1 s^2) & c^2(1-2p_1 c^2)\end{pmatrix}.
\end{equation}

\Cref{Q1-1D-one-sweep} plots the convergence factor $ \rho \big(\widetilde{ E}_1(p_1,\theta)\big)$ as a function of $(p_1, \theta)$. For fixed $p_1$,
we can see that $\displaystyle\max_{\theta}\rho\big(\widetilde{E}_1(p_1,\theta)\big)$ is achieved at either $\theta=\pm \frac{\pi}{2}$ or $\theta=0$.
There is a unique point $p_1=\frac{2}{3}$, where we observe that
$\Psi_{\Tlow}$ in \cref{inner-function} is nondifferentiable, with
$\rho\big(\widetilde{E}_1(\frac{2}{3},0)\big)=\rho\big(\widetilde{E}_1(\frac{2}{3},\frac{\pi}{2})\big)$. For values of $p_1>1$, divergence is observed.
The behavior seen also supports the importance of an appropriate discretization in $\theta$: \Cref{Q1-1D-one-sweep} (right) illustrates that a coarse sampling with an even-valued $\Ntheta$ will result in an approximation \cref{inner-function-2} whose minimum value does not accurately reflect the true convergence factor.

To derive the analytical solution of this minimax problem for $p_1\leq 1$, we note that the two eigenvalues of $\widetilde{E}_1$ are $\lambda_1= 0$ and $\lambda_2 = 1-4p_1\Big((s^2-\frac{1}{2})^2+\frac{1}{4}\Big)$.
Since $\theta\in[-\frac{\pi}{2},\frac{\pi}{2}]$ and $s^2 =\sin^2(\frac{\theta}{2}) \in [0,\frac{1}{2}]$,  it follows that $\max_{\theta} |\lambda_2| = \max\{|1-2p_1|, |1-p_1|\}$. Thus,
\begin{eqnarray*}
\rho_{\rm opt} &=& \min_{0\leq p_1 \leq 1}\max_{\theta \in[-\frac{\pi}{2},\frac{\pi}{2}]}\Big\{\rho\big(\widetilde{E}_1(p_1, \theta)\big)\Big\}\nonumber \\
  &=&\min_{0\leq p_1 \leq 1} \max_{\theta \in[-\frac{\pi}{2},\frac{\pi}{2}]} \Bigg\{\Big|1-4p_1\big((s^2-1/2)^2+1/4\big)\Big|\Bigg\} =\frac{1}{3},
\end{eqnarray*}
if and only if $p_1=\frac{2}{3}$. \revise{This is, of course, the well-known optimal weight for
Jacobi relaxation for this problem, which we compute as a verification of the approach proposed here.}

\begin{figure}
\centering
\includegraphics[width=.48\linewidth]{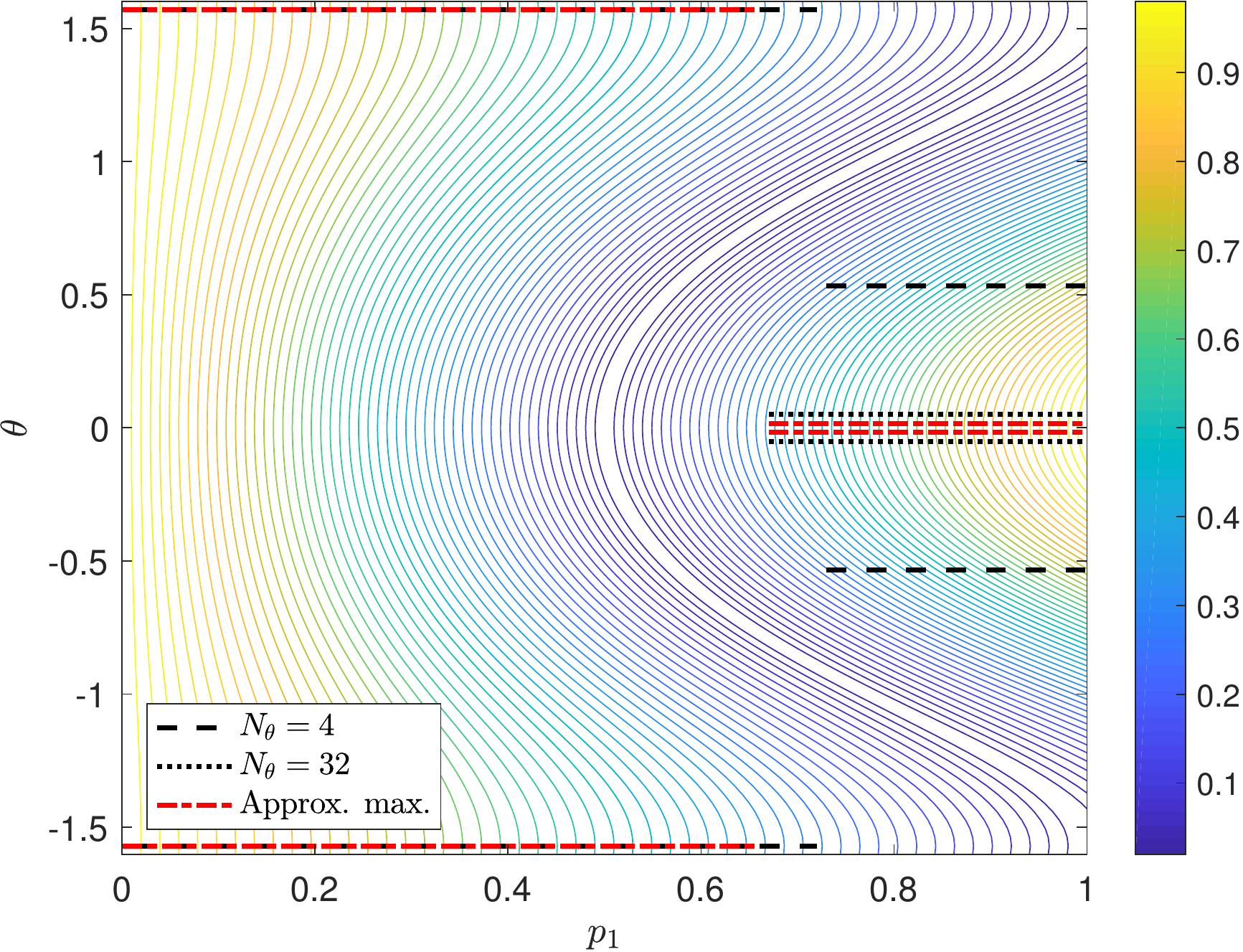} \hfill \includegraphics[width=.48\linewidth]{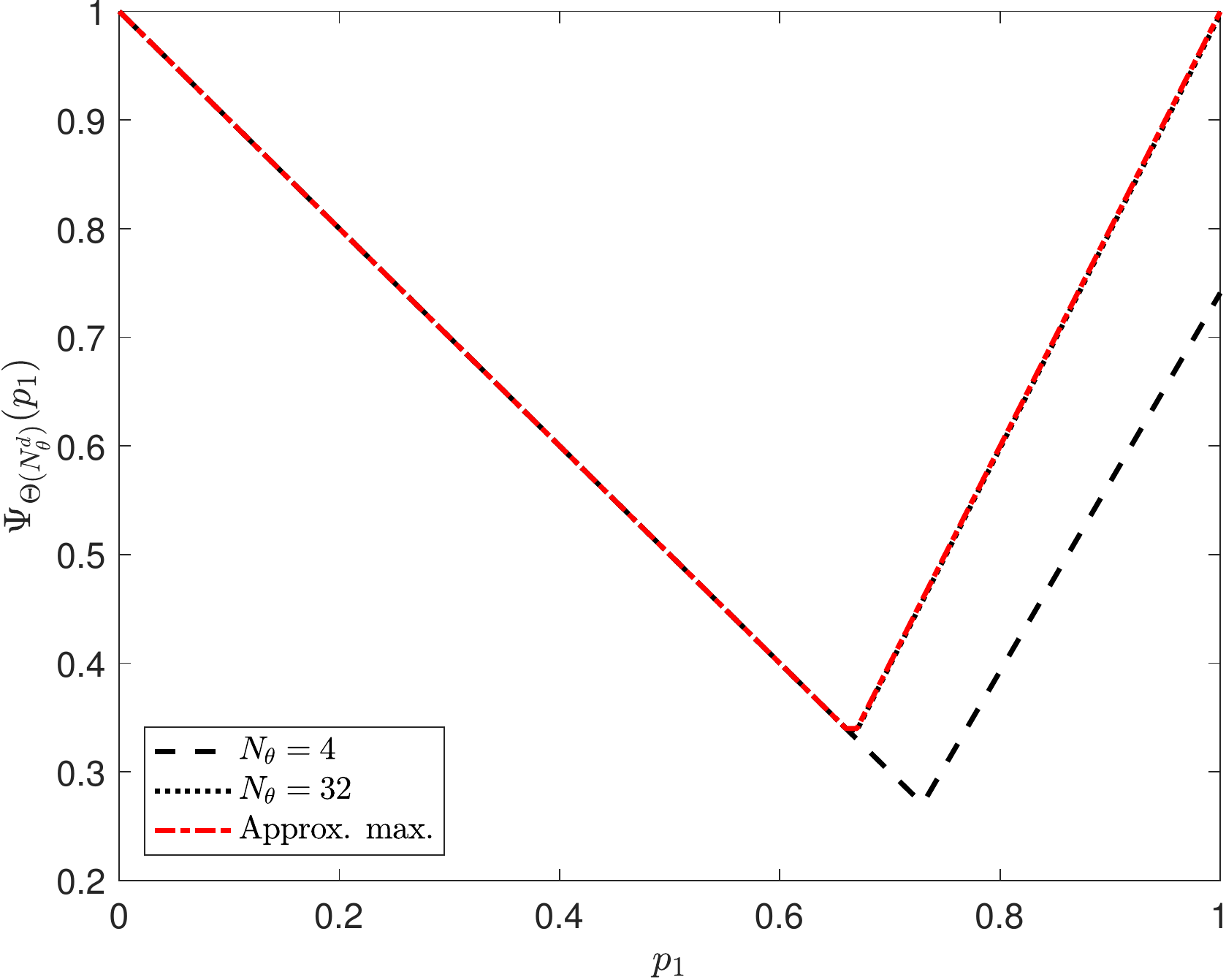}
\caption{Left: Contours of $\rho\big(\widetilde{E}_1(p_1,\theta)\big)$ as a function of $\theta$ and $p_1$ for the $P_1$ discretization of the 1D Laplacian in \cref{eq:Q1Lp}. The overset lines
show the inner-maximizing $\theta$ value
based on different discretizations and approximations in $\theta$.
Right: The corresponding
$\Psi$
 approximations illustrate the differences in the outer function being minimized, with all cases exhibiting nondifferentiability.
}\label{Q1-1D-one-sweep}
\end{figure}

\Cref{Odd-compare-LP1} shows the performance of \ROBOBOA\ and \HANSO\ variants (using analytical derivatives) applied to this optimization problem.
All variants shown achieve near-optimal performance in 100--400 function evaluations, and the two performance metrics show excellent agreement.
For comparison, the brute-force discretized search of \Cref{alg:brute-force} (with $N_p=20$ evenly spaced points in $[0,1]$ and $\Ntheta=32$) takes 640 function evaluations and achieves $\rho_{\Psi_*}=0.35$ with $p_1=0.65$.
Any sampling strategy that happens to sample the points $(\frac{2}{3},0)$ or $(\frac{2}{3},\pm \frac{\pi}{2})$ should, of course, provide the exact solution in this simple case. Because the optimal points occur at $\theta\in \{0,\pm\frac{\pi}{2}\}$, \HANSO\ can successfully achieve the optimum at low cost using coarse sampling in $\theta$ (in this case, with $\Ntheta=3$) that includes these points.  Note that, for this problem, the LFA symbol offers a near-perfect approximation of the expected performance, as can be seen comparing the left and right panels of \Cref{Odd-compare-LP1}.

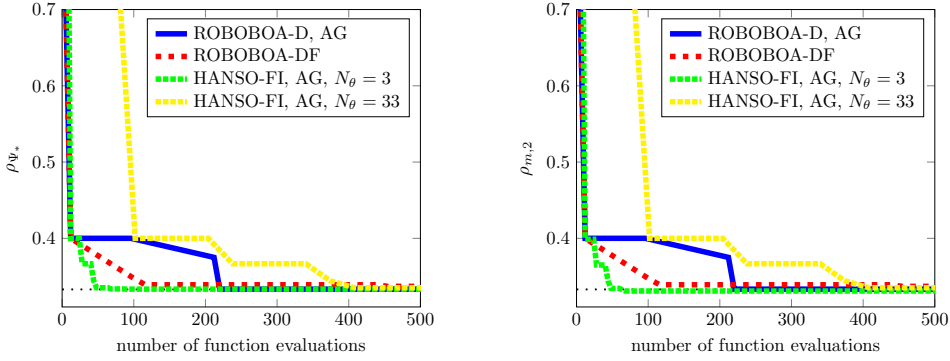
\begin{figure}
\centering
\begin{minipage}[b]{0.475\textwidth}
\centering
\begin{tikzpicture}[baseline, scale=0.695]
    \begin{axis}[xlabel = {number of function evaluations},
                         ylabel = {$\rho_{\Psi_*}$},
                         legend pos = north east,
                        ymin=0.31, ymax=0.7,
                        xmin =0, xmax =500,
                        legend cell align={left}
                             ]
	\draw[line width=1pt, loosely dotted, black] (0,0.333) -- (4000,0.333);
        \addplot  [line width=3pt,solid, blue]  table[x index=0, y index =1]{Q1SNH331.dat};
        \addplot  [line width=3pt,loosely dotted, red]   table[x index=0, y index =1]{Q1SNH332.dat};
        \addplot [line width=3pt, densely dash dot, green] table[x index=0,y index =1]{Q1SNH33.dat};
        \addplot [line width=3pt, densely dash dot, yellow] table[x index=0,y index =1]{Q1SNH333.dat};
        \addlegendentry{ROBOBOA-D, AG}
        \addlegendentry{ROBOBOA-DF}
        \addlegendentry{HANSO-FI, AG, $\Ntheta=3$}
        \addlegendentry{HANSO-FI, AG, $\Ntheta=33$}
\end{axis}
\end{tikzpicture}
\end{minipage}
\hfill
\begin{minipage}[b]{0.475\textwidth}
\centering
\begin{tikzpicture}[baseline, scale=0.695]
    \begin{axis}[xlabel = {number of function evaluations},
                         ylabel = {$\rho_{m,2}$},
                         legend pos = north east,
                        ymin=0.31, ymax=0.7,
                        xmin =0, xmax =500,
                        legend cell align={left}
                             ]
	\draw[line width=1pt, loosely dotted, black] (0,0.333) -- (4000,0.333);
        \addplot  [line width=3pt,solid, blue]  table[x index=0, y index =1]{Q1SVNH331.dat};
        \addplot  [line width=3pt,loosely dotted, red]   table[x index=0, y index =1]{Q1SVNH332.dat};
        \addplot [line width=3pt, densely dash dot, green] table[x index=0,y index =1]{Q1SVNH33.dat};
        \addplot [line width=3pt, densely dash dot, yellow] table[x index=0,y index =1]{Q1SVNH333.dat};
        \addlegendentry{ROBOBOA-D, AG}
        \addlegendentry{ROBOBOA-DF}
        \addlegendentry{HANSO-FI, AG, $\Ntheta=3$}
        \addlegendentry{HANSO-FI, AG, $\Ntheta=33$}
\end{axis}
\end{tikzpicture}
\end{minipage}
\caption{Optimization performance for Laplace in 1D with a single relaxation,  \cref{eq:Q1Lp}.
Left: predicted performance using \cref{eq:rho_psi_metric};  right: measured two-grid performance using $\rho_{m,2}$ (see \cref{eq:defi-rho2}).}\label{Odd-compare-LP1}
\end{figure}

If we consider $\nu_1=\nu_2=1$ and $p_1=\frac{2}{3}$, then $\rho_{\rm opt}=(\frac{1}{3})^2=\frac{1}{9}$.  A natural question is  whether we can improve the two-grid performance by using two different weights for pre- and post- relaxation. Thus, we consider
\begin{equation}
\label{eq:E2Laplace}
 \widetilde{E}_2 =\widetilde{\boldsymbol{ S}}_2\widetilde{\boldsymbol{ Q}}\widetilde{\boldsymbol{ S}}_1.
\end{equation}
For this problem,  $\lambda\big(\widetilde{E}_2\big) =\lambda\big(\widetilde{\boldsymbol{ S}}_2\widetilde{\boldsymbol{ Q}}\widetilde{\boldsymbol{ S}}_1\big) =\lambda\big(\widetilde{\boldsymbol{ Q}}\widetilde{\boldsymbol{ S}}_1\widetilde{\boldsymbol{ S}}_2\big)$; and, similar to the calculation above, we have
\begin{equation*}
  \widetilde{\boldsymbol{ Q}}\widetilde{\boldsymbol{ S}}_1\widetilde{\boldsymbol{ S}}_2 = \begin{pmatrix} (1-2p_1s^2)(1-2p_2s^2)s^2 &  -(1-2p_1c^2)(1-2p_2c^2)c^2\\
  -(1-2p_1s^2)(1-2p_2s^2)s^2 & (1-2p_1c^2)(1-2p_2c^2)c^2
   \end{pmatrix}.
\end{equation*}
The two eigenvalues of $\widetilde{E}_2$ are $\lambda_1=0$ and $\lambda_2=(1-2p_1s^2)(1-2p_2s^2)s^2+(1-2p_1c^2)(1-2p_2c^2)c^2$, which we note is symmetric about $p_1=p_2$. Therefore
\begin{eqnarray*}
\rho_{\rm opt} &=& \min_{\pb \in \Rp^2} \max_{\theta \in[-\frac{\pi}{2},\frac{\pi}{2}]}\Big\{\rho\big(\widetilde{ E}_2(\pb,\theta)\big)\Big\} \nonumber\\
  &=&\min_{\pb \in \Rp^2} \max_{\theta \in[-\frac{\pi}{2},\frac{\pi}{2}]} \Big\{\big|(1-2p_1s^2)(1-2p_2s^2)s^2+(1-2p_1c^2)(1-2p_2c^2)c^2\big|\Big\} = 0,
\end{eqnarray*}
with the
last equality obtained (independent of $\theta$)
if and only if $(p_1,p_2)$ is $(1,\frac{1}{2})$ or its symmetric equivalent $(\frac{1}{2},1)$.  \revise{Since $\tilde{E}_2$ is not the zero matrix, a zero spectral radius indicates convergence to the exact solution in at most two iterations of the two-grid method; while this result can be verified algebraically, to our knowledge it is new to the literature.}

The inner function of $\pb$ (i.e., \cref{inner-function}) is illustrated on the top right of \Cref{fig:newQ1-1D-two-sweep}, wherein the nondifferentiability with respect to $\pb$ is evident along two quadratic level curves. Approximations \cref{inner-function-2} of this inner function are shown in the top left ($\Ntheta=2$) and top middle ($\Ntheta=4$) plots of \Cref{fig:newQ1-1D-two-sweep}. For each case, the bottom row shows the corresponding inner (approximately) maximizing $\theta$ value.



\begin{figure}
\centering
\includegraphics[width=0.32\linewidth]{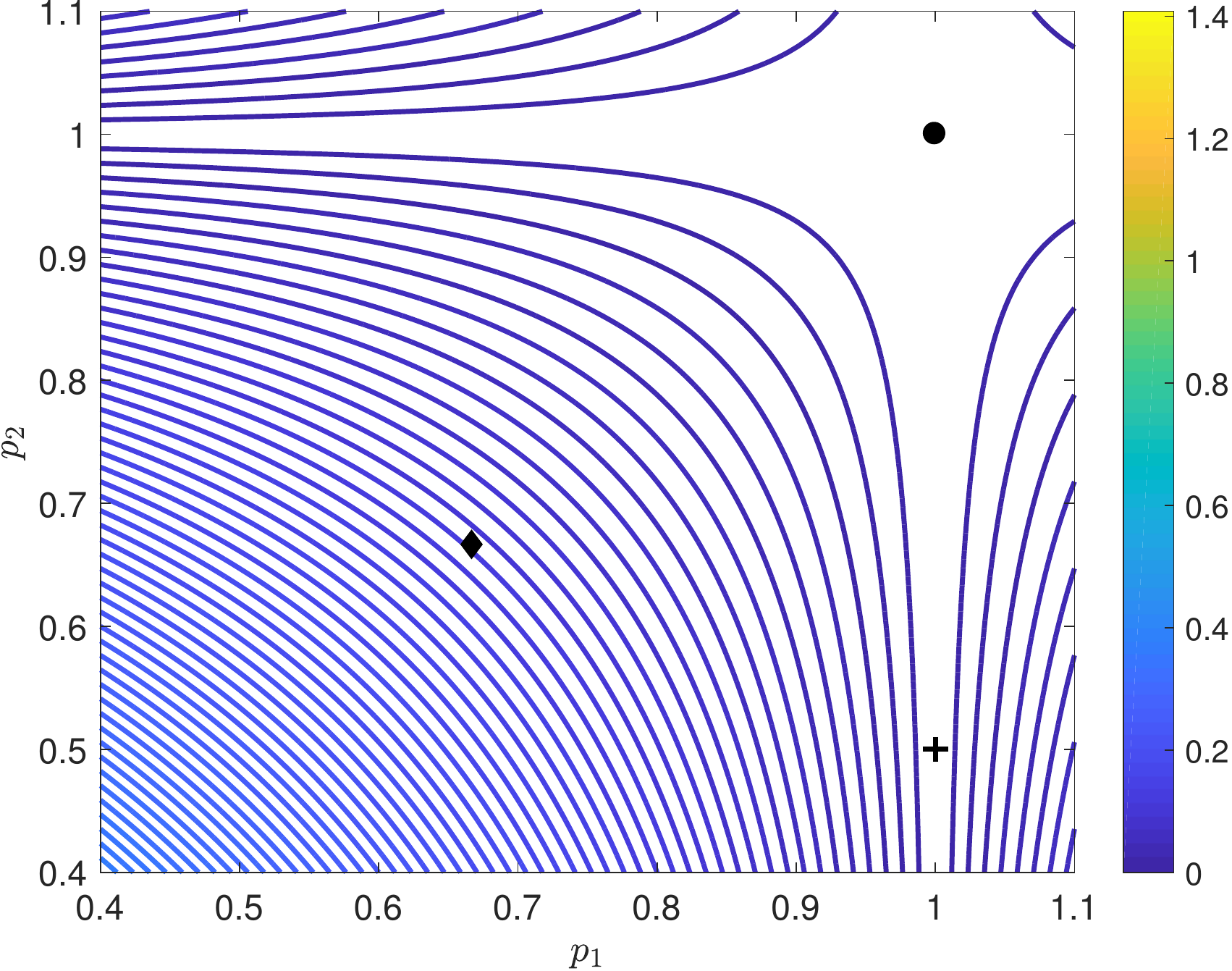} \hfill
\includegraphics[width=0.32\linewidth]{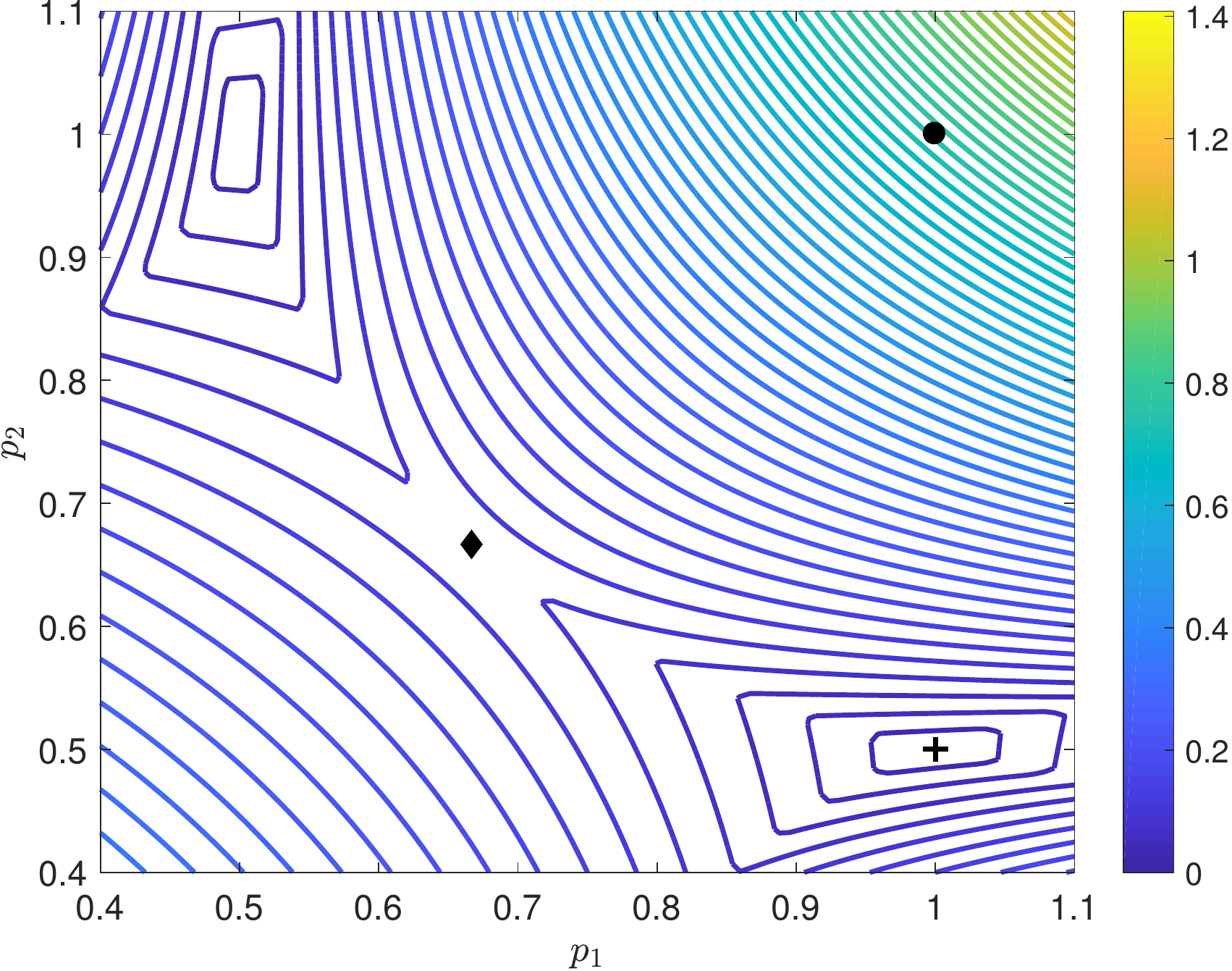}  \hfill
\includegraphics[width=0.32\linewidth]{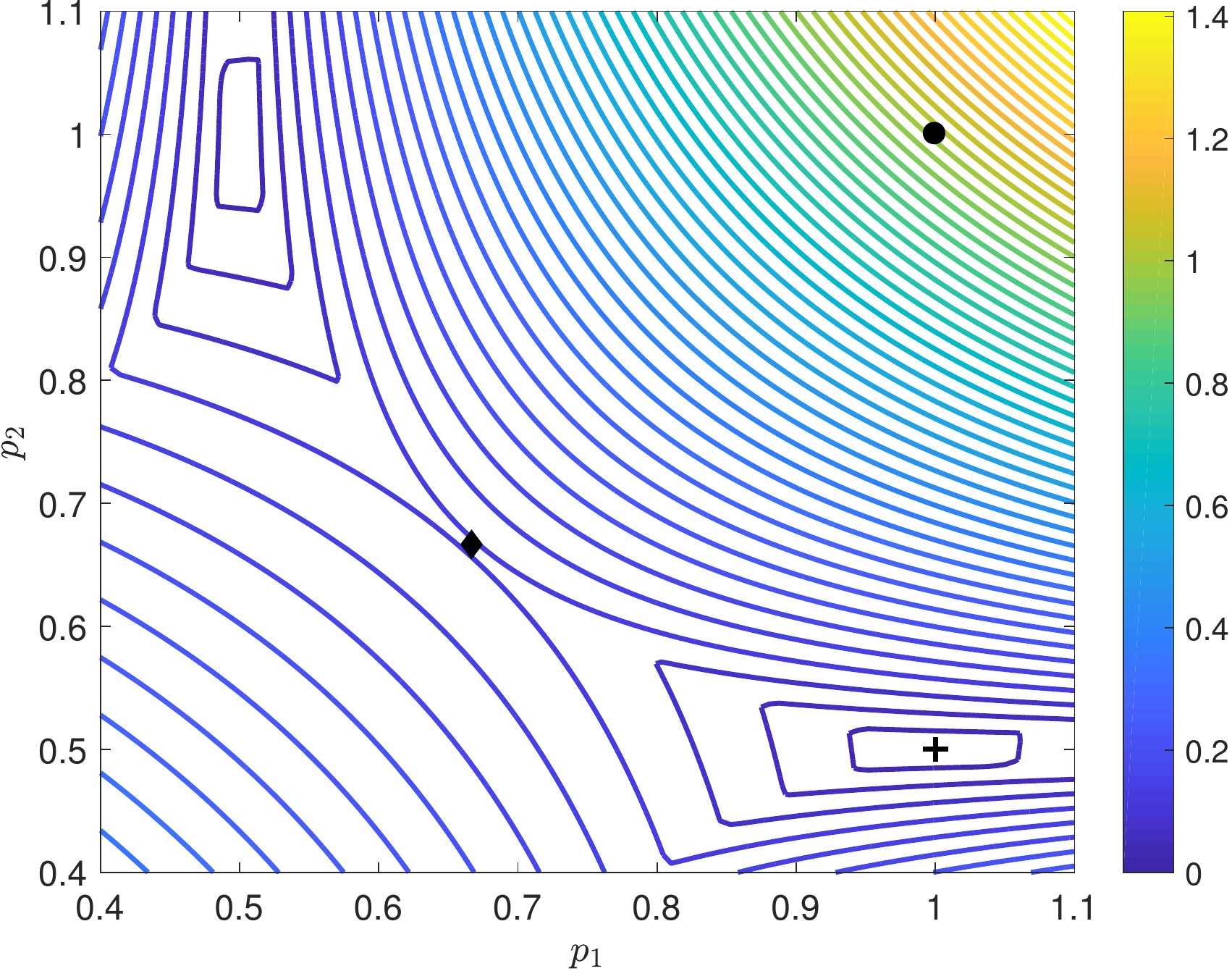}

\includegraphics[width=0.32\linewidth]{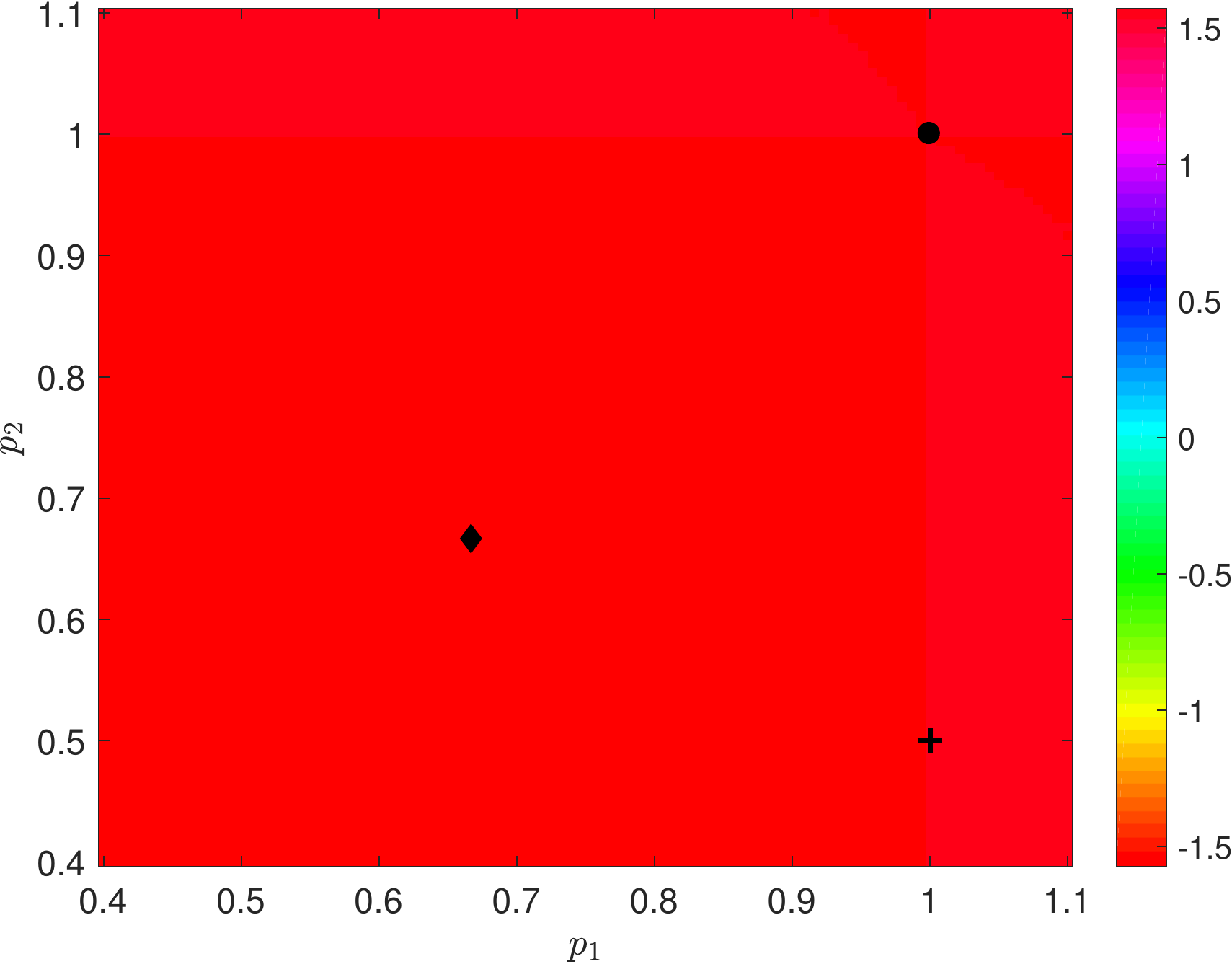} \hfill
\includegraphics[width=0.32\linewidth]{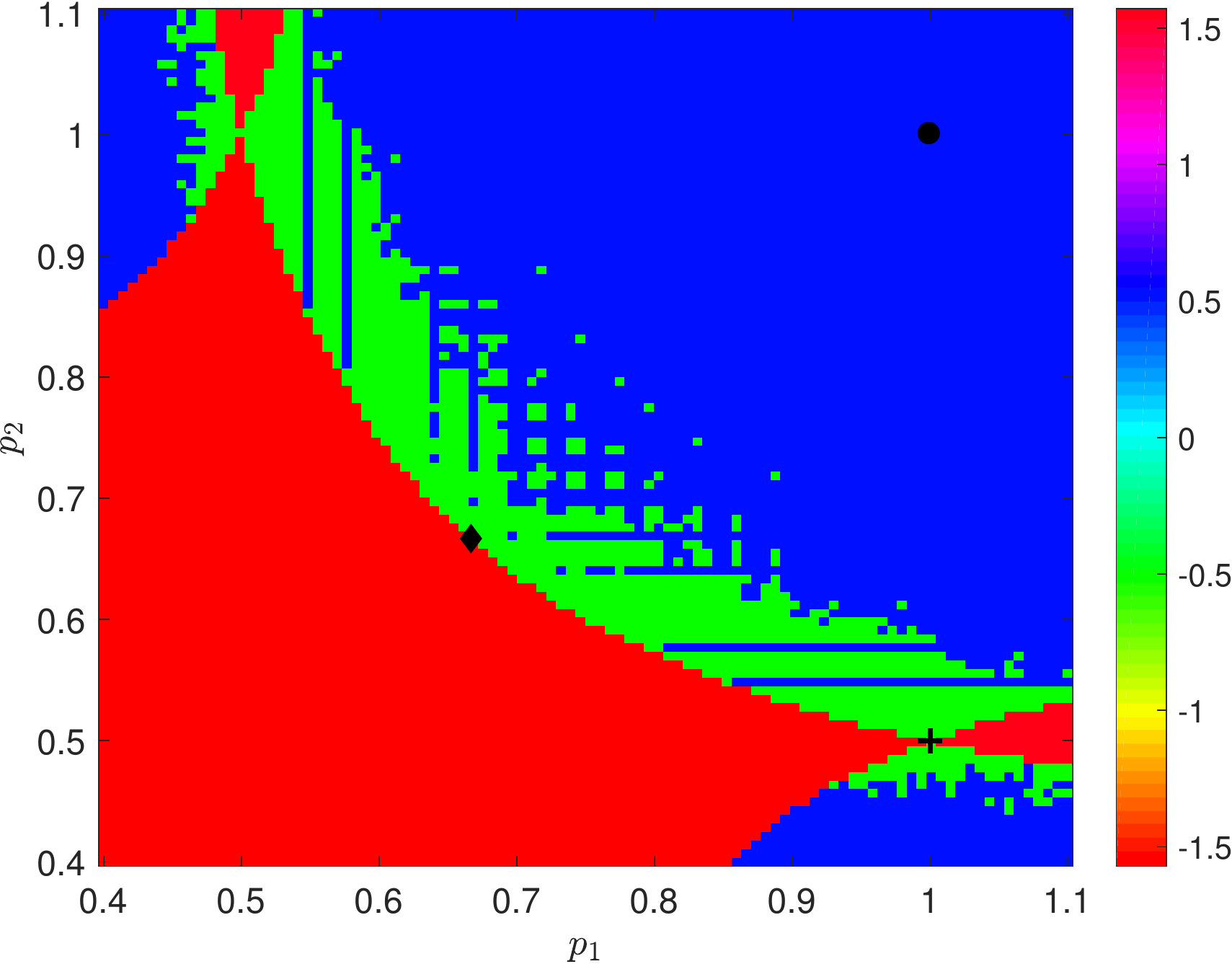}  \hfill
\includegraphics[width=0.32\linewidth]{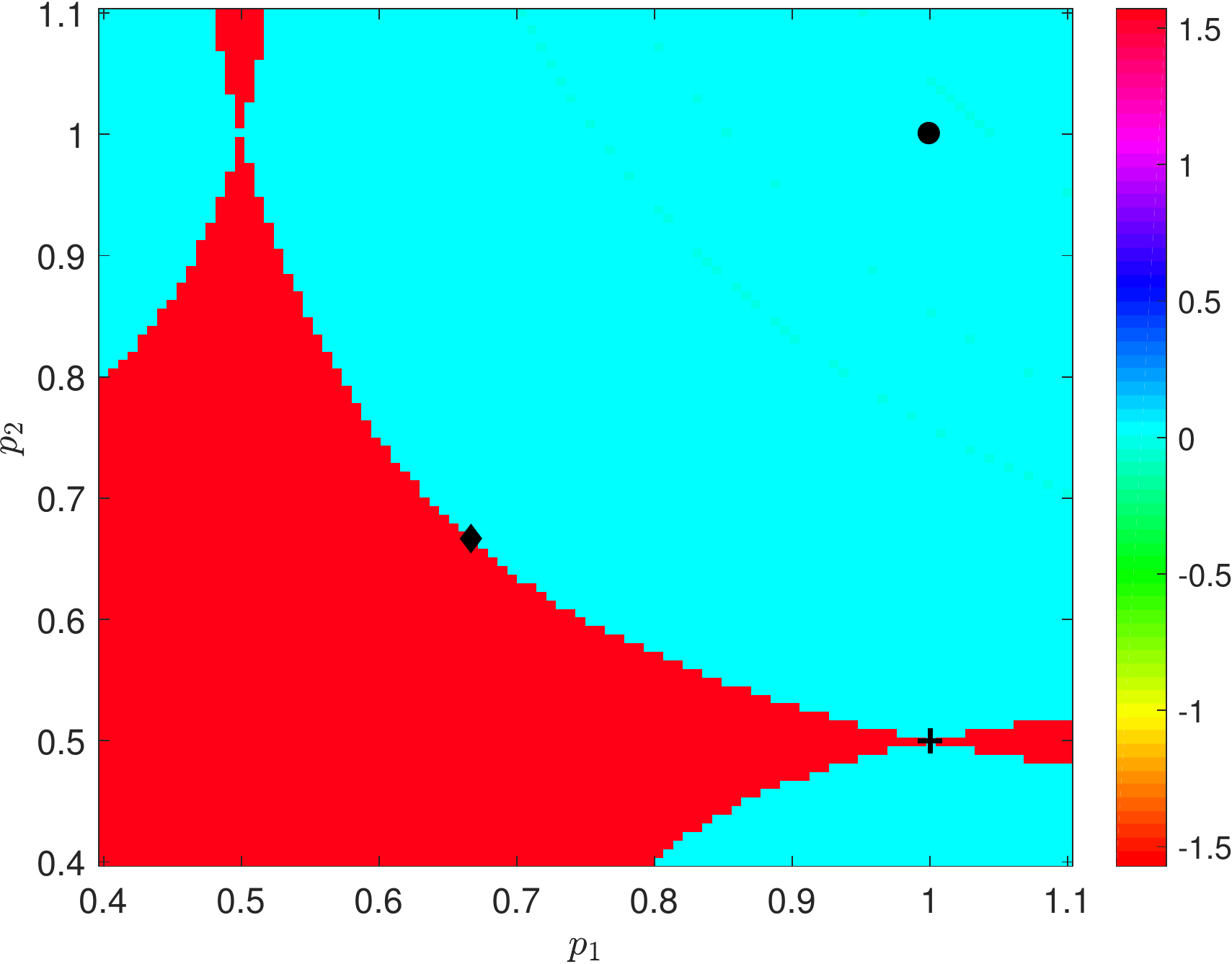}

\caption{Using 2 sweeps of Jacobi relaxation for $P_1$-Laplace in 1D (i.e., \cref{eq:E2Laplace}) as a function of $(p_1,p_2)$.
Top row: Contours of \cref{inner-function-2} using resolution levels (from left to right) $\Ntheta=2, \Ntheta=4$, and approximately solving \eqref{max-problem}.
Cross denotes the solution obtained by the \ROBOBOA\ variants, diamond denotes the solution obtained by \HANSO\ for $\Ntheta\geq 4$ resolutions, circle denotes
the solution obtained by \HANSO\ for $\Ntheta=2$ resolution.
Bottom row: Corresponding
maximizer $\theta$ in each approximation.}\label{fig:newQ1-1D-two-sweep}
\end{figure}

\Cref{Odd-compare-LP2-NH4} shows the performance (left: the LFA-predicted convergence factor using a fixed sampling of Fourier space, $\rho_{\Psi_*}$, as in \cref{eq:rho_psi_metric} with $\Ntheta=33$; right: measured two-grid convergence, $\rho_{m,1}$) of the three methods for this problem using analytical (or no) derivatives. We see that using \ROBOBOA\ with derivatives is successful, finding a good approximation to $\rho_{\rm opt}$ using fewer than 100 $\rho$  evaluations. \ROBOBOA\ without derivatives achieves a similar approximate solution but at a greater expense, requiring roughly 900 function evaluations to find a comparable approximation.
In contrast, \HANSO\ converges to a value of $\rho_{\Psi_*}\approx 0.11$.
The corresponding approximate solutions are illustrated in \Cref{fig:newQ1-1D-two-sweep}, which shows that \HANSO\ (in fact, independent of $\Ntheta$-sampling resolution once $N_\theta> 2$) is converging to a Clarke-stationary saddle point.
\revise{Because we employed the gradient sampling method within \HANSO, this observed convergence behavior is theoretically supported \cite{burke2005robust}.}
We again note good agreement between the LFA-predicted and measured two-grid performance.

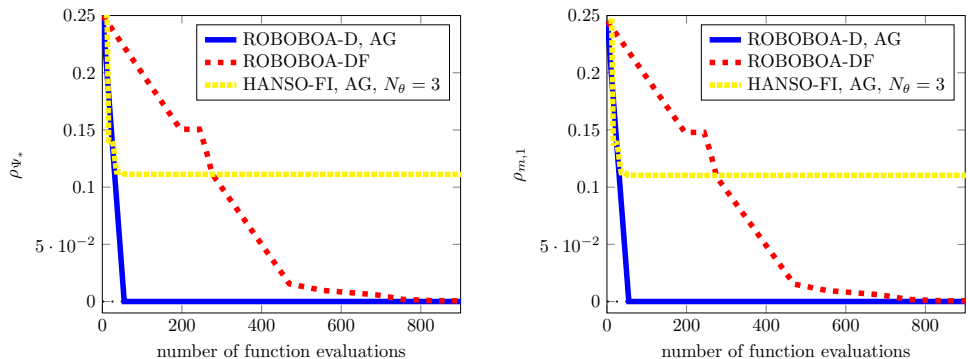
\begin{figure}
\centering
\begin{minipage}[b]{0.485\textwidth}
\centering
\begin{tikzpicture}[baseline, scale=0.695]
    \begin{axis}[xlabel = {number of function evaluations},
                         ylabel = {$\rho_{\Psi_*}$},
                         legend pos = north east,
                        ymin=-0.01, ymax=0.25,
                        xmin =0, xmax =900,
                        legend cell align={left}
                             ]
	\draw[line width=1pt, loosely dotted, black] (0,0.) -- (4000,0.);
        \addplot  [line width=3pt, solid, blue]  table[x index=0, y index =1]{Q1DNH31.dat};
        \addplot  [line width=3pt, loosely dotted, red]   table[x index=0, y index =1]{Q1DNH32.dat};
        \addplot [line width=3pt, densely dash dot, yellow] table[x index=0,y index =1]{Q1DNH33.dat};
        \addlegendentry{ROBOBOA-D, AG}
        \addlegendentry{ROBOBOA-DF}
        \addlegendentry{HANSO-FI, AG, $\Ntheta=3$}
\end{axis}
\end{tikzpicture}
\end{minipage}
\hfill
\begin{minipage}[b]{0.485\textwidth}
\centering
\begin{tikzpicture}[baseline, scale=0.695]
    \begin{axis}[xlabel = {number of function evaluations},
                         ylabel = {$\rho_{m,1}$},
                         legend pos = north east,
                        ymin=-0.01, ymax=0.25,
                        xmin =0, xmax =900,
                        legend cell align={left}
                             ]
	\draw[line width=1pt, loosely dotted, black] (0,0.) -- (4000,0.);
        \addplot  [line width=3pt, solid, blue]  table[x index=0, y index =1]{Q1DVNH31.dat};
        \addplot  [line width=3pt, loosely dotted, red]   table[x index=0, y index =1]{Q1DVNH32.dat};
        \addplot [line width=3pt, densely dash dot, yellow] table[x index=0,y index =1]{Q1DVNH33.dat};
        \addlegendentry{ROBOBOA-D, AG}
        \addlegendentry{ROBOBOA-DF}
        \addlegendentry{HANSO-FI, AG, $\Ntheta=3$}
\end{axis}
\end{tikzpicture}
\end{minipage}
\caption{Optimization performance for Laplace in 1D using the $P_1$ discretization with two sweeps of Jacobi relaxation \cref{eq:E2Laplace} measured using LFA-predicted $\rho_{\Psi_*}$ (left) and measured $\rho_{m,1}$ (right).}\label{Odd-compare-LP2-NH4}
\end{figure}

Next, we consider the two-grid method for the $P_2$ approximation for the Laplace problem in 2D with a single weighted Jacobi relaxation, with $\nu_1+\nu_2=1$. In 2D, there are four harmonics:
$(\theta_1,\theta_2)\in T^{\rm low}$, $(\theta_1 + \pi,\theta_2)$, $(\theta_1,\theta_2+\pi)$, and $(\theta_1+\pi,\theta_2+\pi)$.  For $P_2$ elements, we have 4 types of basis functions, leading to a $4\times 4$ matrix symbol for each harmonic.  Thus, $\widetilde{E}$ is a $16\times 16$ matrix; for more details on computing $\widetilde{E}$, we refer to \cite{StokePatchFHM}.
For this example, we do not know the analytical solution;
brute-force discretized search with
$N_p=50$ points in parameter $p_1\in[0,1]$ and $\Ntheta=33$ points in each component of $\thetab$, we find $\rho_{\Psi_*}\approx 0.596$ with parameter $p_1\approx 0.82$ using 54,450 function evaluations.

  Using central differences for the derivative approximation, the left of \Cref{odd-compare-LP3} shows the performance  of the three optimization approaches for this problem, using a central-difference stepsize $t=10^{-8}$. All three approaches are successful, yielding something comparable to the brute-force discretized value within 500--600 function evaluations. Because we can sample coarsely in $\thetab$, \HANSO\ (with $\Ntheta=3$) is most efficient in this setting. Comparing the left and right of \Cref{odd-compare-LP3}, we see that the  measured convergence factor is slightly better than the discretized $\rho_{\Psi_*}$ prediction of the three methods; this is reasonable, since LFA offers a sharp
prediction.

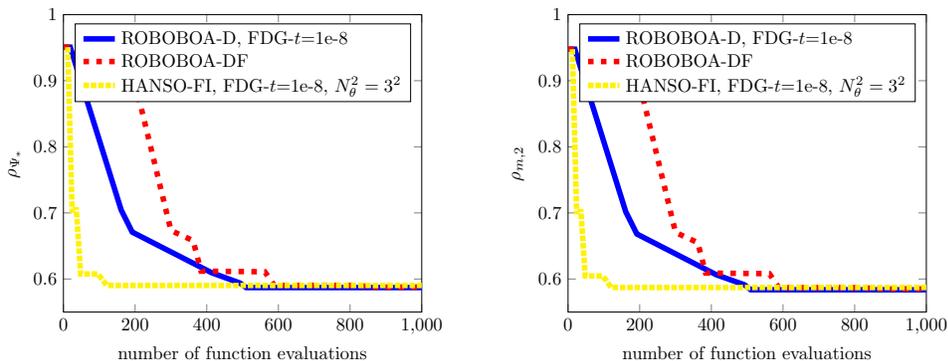
\begin{figure}
\centering
\begin{minipage}[b]{0.485\textwidth}
\centering
\begin{tikzpicture}[baseline, scale=0.695]
    \begin{axis}[xlabel = {number of function evaluations},
                         ylabel = {$\rho_{\Psi_*}$},
                         legend pos = north east,
                        ymin=0.55, ymax=1,
                        xmin =0, xmax =1000,
                        legend cell align={left}
                             ]
        \addplot  [line width=3pt, solid, blue]  table[x index=0, y index =1]{P2NH3-81.dat};
        \addplot  [line width=3pt, loosely dotted, red]   table[x index=0, y index =1]{P2NH3-82.dat};
        \addplot [line width=3pt, densely dash dot, yellow] table[x index=0,y index =1]{P2NH3-83.dat};

        \addlegendentry{ROBOBOA-D, FDG-$t$=1e-8}
        \addlegendentry{ROBOBOA-DF}
        \addlegendentry{HANSO-FI, FDG-$t$=1e-8, $\Ntheta^2=3^2$}
\end{axis}
\end{tikzpicture}
\end{minipage}
\hfill
\begin{minipage}[b]{0.485\textwidth}
\centering
\begin{tikzpicture}[baseline, scale=0.695]
    \begin{axis}[xlabel = {number of function evaluations},
                         ylabel = {$\rho_{m,2}$},
                         legend pos = north east,
                        ymin=0.55, ymax=1,
                        xmin =0, xmax =1000,
                        legend cell align={left}
                             ]
        \addplot  [line width=3pt, solid, blue]  table[x index=0, y index =1]{P2VNH3-81.dat};
        \addplot  [line width=3pt, loosely dotted, red]   table[x index=0, y index =1]{P2VNH3-82.dat};
        \addplot [line width=3pt, densely dash dot, yellow] table[x index=0,y index =1]{P2VNH3-83.dat};
        \addlegendentry{ROBOBOA-D, FDG-$t$=1e-8}
        \addlegendentry{ROBOBOA-DF}
        \addlegendentry{HANSO-FI, FDG-$t$=1e-8, $\Ntheta^2=3^2$}
\end{axis}
\end{tikzpicture}
\end{minipage}
\caption{Optimization performance for $P_2$ discretization of the Laplace equation in 2D using central-difference derivatives with stepsize $t=10^{-8}$
using LFA-predicted $\rho_{\Psi_*}$ (left) and measured $\rho_{m,2}$ (right).}\label{odd-compare-LP3}
\end{figure}

\begin{table}
 \caption{Two-grid LFA-predicted ($\rho_{\Psi_*}$)  and measured ($\rho_{m,1}$ and $\rho_{m,2}$) convergence factors for the $P_2$
   approximation of the Laplacian at $p_1=0.82$.}
\centering
\begin{tabular}{|l||c|c|c|c|}
\hline
  & $TG(0,1)$   &$TG(1,1)$   & $TG(1,2)$  &$TG(2,2)$     \\
\hline \hline
$\rho_{\Psi_*}$  &0.596     &0.516   &0.391     &0.312   \\
\hline
$\rho_{m,1}$      &0.577    &0.496    &0.360     &0.300   \\
\hline
$\rho_{m,2}$      &0.593    &0.507    &0.370     &0.309   \\
\hline
\end{tabular}\label{odd-P2-Laplace-LFA-vs-MG}
\end{table}

As a validation test, we report the LFA predictions and the measured convergence factors using $p_1=0.82$  for different numbers of pre- and post- smoothing $(\nu_1,\nu_2)$ of the two-grid method ($TG(\nu_1,\nu_2)$) in \Cref{odd-P2-Laplace-LFA-vs-MG}. Consistent with the behavior seen in \Cref{odd-compare-LP3}, in \Cref{odd-P2-Laplace-LFA-vs-MG} we see a very small
difference between the predicted and measured convergence factors, providing further confidence in our prediction $\rho_{\Psi_*}$. Moreover, the LFA-based $\rho_{\Psi_*}$ offers a good prediction of both $\rho_{m,1}$ and $\rho_{m,2}$.
Consequently, in the following examples
we measure performance using only $\rho_{\Psi_*}$.

\revise{
  As a final example of applying these optimization algorithms to a scalar differential equation, we return to the 1D $P_1$ discretization of the Laplacian, but now consider coarsening-by-three multigrid with piecewise constant interpolation operator with stencil representation
\begin{equation*}
P_h=
\left ]
  \begin{tabular}{ccc}
  1 \\
  1  \\
  1
  \end{tabular}
\right [_{3h}^h,
\end{equation*}
$R_h=P^T$, and $L_H = R_hL_hP_h$.  We use a single sweep of pre- and post- Jacobi relaxation with weight $p_1$ and, to attempt to ameliorate the choice of interpolation operator, introduce a second weighting parameter, $p_2$, to under- or over-damp the coarse-grid correction process. The resulting two-grid error propagation operator is
\begin{equation*}
 E =(I -p_1 M_h^{-1}L_h)(I - p_2P_hL_H^{-1}R_hL_h)(I -p_1 M_h^{-1}L_h).
\end{equation*}
We omit the computation of the symbols in this case, but note that similar computations can be found, for example, in \cite{gaspar2009geometric}. When coarsening by threes, the low frequency range becomes $T^{\rm low}=[-\frac{\pi}{3},\frac{\pi}{3})$, and there are 3 harmonic frequencies for each low-frequency mode; thus, the symbol of $E$ is a $3\times 3$ matrix.}

\revise{
For this example, an analytical solution for the optimal parameters is not known; brute-force discretized search with
$N_p=126^2$ points for parameters $\pb\in[0,2.5]^2$ and $\Ntheta=33$ points for $\thetab$ yields $\rho_{\Psi_*}\approx 0.421$ with parameter $(p_1, p_2) = (0.72, 2.30)$.  This has been confirmed using measured convergence factors for both periodic and Dirichlet boundary conditions.
Using central differences for the derivative approximation, \Cref{odd-compare-LP1-coarsen3} shows the performance  of the three optimization approaches for this problem, using a central-difference stepsize $t=10^{-8}$. All three approaches are successful, yielding something comparable to the brute-force discretized value within 500 function evaluations, with no further improvement observed, even when increasing the computational budget to 1500 function evaluations.  For this example, we note that ROBOBOA-DF gets closest to the value found by the brute-force search, while ROBOBOA-D is slightly worse (0.442 vs. 0.429).  While HANSO-FI with $N_\theta = 3$ appears to be equally effective to ROBOBOA based on the small sampling space, the parameter values found are slightly suboptimal when higher-resolution sampling in $\theta$ is considered, giving a convergence factor of 0.461.}

\begin{figure}
\centering
\begin{minipage}[b]{0.485\textwidth}
\centering
\begin{tikzpicture}[baseline, scale=0.695]
    \begin{axis}[xlabel = {number of function evaluations},
                         ylabel = {$\rho_{\Psi_*}$},
                         legend pos = north east,
                        ymin=0.40, ymax=1,
                        xmin =0, xmax =800,
                        legend cell align={left}
                             ]
        \addplot  [line width=3pt, solid, blue]  table[x index=0, y index =1]{Q1C3NH3G1.dat};
        \addplot  [line width=3pt, loosely dotted, red]   table[x index=0, y index =1]{Q1C3NH3G2.dat};
        \addplot [line width=3pt, densely dash dot, yellow] table[x index=0,y index =1]{Q1C3NH3G3.dat};

        \addlegendentry{ROBOBOA-D, FDG-$t$=1e-8}
        \addlegendentry{ROBOBOA-DF}
        \addlegendentry{HANSO-FI, FDG-$t$=1e-8, $\Ntheta=3$}
\end{axis}
\end{tikzpicture}
\end{minipage}
\caption{Optimization performance for $P_1$ discretization of the Laplace equation in 1D  with coarsening-by-threes and piecewise constant interpolation. The derivatives are obtained using central differences with  stepsize $10^{-8}$.}\label{odd-compare-LP1-coarsen3}
\end{figure}
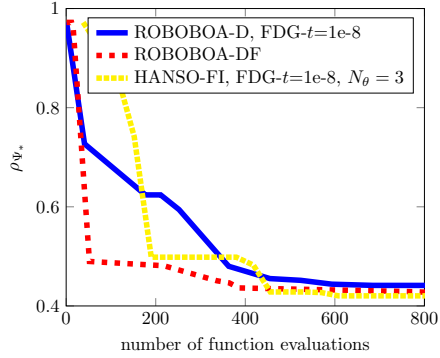

\subsection{Stokes Equations}\label{sec:stokes}
LFA has been applied to both finite-difference and finite-element methods for the Stokes equations, with several different relaxation schemes \cite{StokePatchFHM,NLA2147,luo2018monolithic,HMFEMStokes,SPMacLachlan_CWOosterlee_2011a,drzisga2018analysis,gaspar2014simple,niestegge1990analysis,sivaloganathan1991use}.  Here, we present a variety of examples for the choice of relaxation within the two-grid method using a single pre-relaxation (i.e., $\nu_1=1$ and $\nu_2=0$).  While we do not review the details of LFA for these schemes (see the references above), we briefly review the discretization and relaxation schemes.

 We consider the Stokes equations,
\begin{equation}
\label{eq:Stokes-P2P1}
\begin{array}{r}
  -\Delta\vec{u}+\nabla \mathfrak{p} = \vec{f}\\
  -\nabla\cdot \vec{u}=0,
\end{array}
\end{equation}
where $\vec{u}$ is the velocity vector, $\mathfrak{p}$ is the  scalar pressure of a viscous fluid, and $\vec{f}$ represents a known forcing term, together with suitable boundary conditions.  Discretizations of \cref{eq:Stokes-P2P1} typically lead to  linear systems of the form
\begin{equation}\label{saddle-structure-P2P1}
     L_hy=\begin{pmatrix}
      A & B^{T}\\
     B & - C\\
    \end{pmatrix}
        \begin{pmatrix} \mathcal{U }\\ {\rm p}\end{pmatrix}
  =\begin{pmatrix} {\rm f} \\ 0 \end{pmatrix}=b,
 \end{equation}
where $A$ corresponds to the discretized vector Laplacian and $B$ is the negative of the discrete divergence operator.
If the discretization is naturally unstable, such as for the $Q_1-Q_1$ discretization, then $C\neq 0$ is the stabilization matrix; otherwise $C=0$, such as  for the stable $P_2-P_1$ and $Q_2-Q_1$ (Taylor-Hood elements) finite-element discretizations \cite{elman2006finite}.

We consider three discretizations. First, we consider the stable staggered finite-difference MAC scheme, with edge-based velocity degrees of freedom and cell-centered pressures. Second, we consider the unstable $Q_1-Q_1$ finite-element discretization on quadrilateral meshes, taking all degrees of freedom in the system to be collocated at the nodes of the mesh, stabilized by taking $C$ to be $h^2$ times the $Q_1$ discretization of the Laplacian. Third, we consider the stable $P_2-P_1$ finite-element discretization on triangular meshes, with $C=0$.

\subsubsection{Finite-Difference Discretizations}\label{sec:fdd}

We first consider Braess-Sarazin relaxation based on solution \revise{of the simplified linear system with damping parameter $p_2$}
\begin{equation}\label{Precondtion}
   M_I\delta x=  \begin{pmatrix}
      p_2 D & B^{T}\\
     B & 0\\
    \end{pmatrix}
    \begin{pmatrix} \delta \mathcal{U} \\ \delta {\rm p}\end{pmatrix}
  =\begin{pmatrix} r_{\mathcal{U}} \\ r_{\rm p}\end{pmatrix},
\end{equation}
where $D={\rm diag}(A)$.
 Solutions of \cref{Precondtion} are computed in two stages as
\begin{equation}\label{solution-of-precondtion}
\begin{array}{rcl}
  (BD^{-1}B^{T})\delta {\rm p}&=&BD^{-1}r_{\mathcal{U}}-p_2 r_{\rm p}, \\
  \delta \mathcal{U}&=&\frac{1}{p_2}D^{-1}(r_{\mathcal{U}}-B^{T}\delta {\rm p}).
\end{array}
\end{equation}
In practice, \cref{solution-of-precondtion} is not solved exactly; an approximate solve is sufficient, such as using a simple sweep of a Gauss-Seidel or weighted Jacobi iteration in place of inverting $BD^{-1}B^T$. For the inexact Braess-Sarazin relaxation scheme, we consider a single sweep of  weighted Jacobi relaxation on $BD^{-1}B^T$ with weight $p_3$ and an outer damping parameter, $p_1$, for the whole relaxation scheme (i.e., $\mathcal{S}_h=I -p_1 M_I^{-1}L_h$). This gives a three-dimensional parameter space, $\pb=(p_1,p_2,p_3)$, for the optimization.

Since each harmonic has a three-dimensional symbol, the resulting two-grid LFA representation $\widetilde{E}_I(\pb,\thetab)$ is a $12\times 12$ system. In \cite{NLA2147},  the solution to the minimax problem is shown to be
\begin{equation}\label{eq:opt-MAC-IBSR}
\rho_{\rm opt} = \min_{\pb \in \Rp^3}\max_{\thetab \in[-\frac{\pi}{2},\frac{\pi}{2}]^2}\left\{\rho \left(\widetilde{E}_{I}\left((p_1,p_2,p_3), \thetab \right)\right)\right\}
  =\frac{3}{5},
\end{equation}
with $(p_1, p_2, p_3)=(1,\frac{5}{4},\frac{4}{5})$, although this parameter choice may not be unique.

Here, we explore the influence of the accuracy of the derivatives on the performance of \HANSO\ and \ROBOBOA\ with derivatives. Analytical derivatives are used in the left of \Cref{odd-compare-IBSR-MAC}, which shows that \ROBOBOA\ with derivatives obtains a near-optimal value of $\rho_{\Psi_*}$ in approximately \revise{half} as many evaluations as \ROBOBOA\ without derivatives. Central-difference derivatives (with $t= 10^{-12}$) are used in the right of \Cref{odd-compare-IBSR-MAC}, which shows that the performance (in terms of function evaluations required) of both \ROBOBOA\ and \HANSO\ suffer.

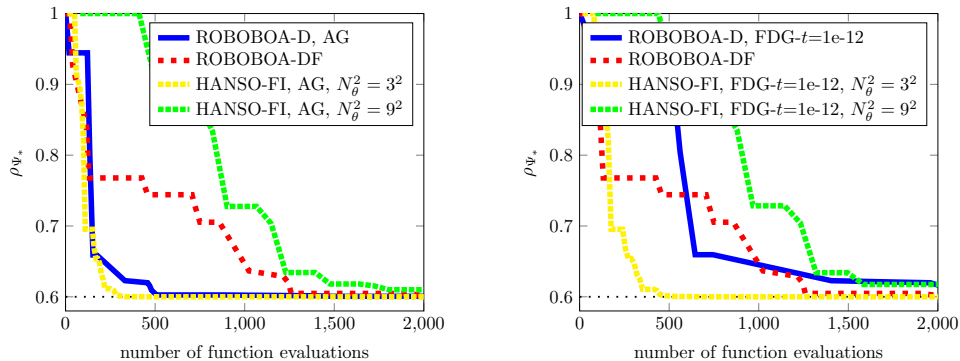
\begin{figure}
\centering
\begin{minipage}[b]{0.475\textwidth}
\centering
\begin{tikzpicture}[baseline, scale=0.695]
    \begin{axis}[xlabel = {number of function evaluations},
                         ylabel = {$\rho_{\Psi_*}$},
                         legend pos = north east,
                        ymin=0.58, ymax=1,
                        xmin =0, xmax =2000,
                        legend cell align={left}
                             ]
	\draw[line width=1pt, loosely dotted, black] (0,0.6) -- (4000,0.6);
        \addplot  [line width=3pt, solid, blue]  table[x index=0, y index =1]{IBSRNH31.dat};
        \addplot  [line width=3pt, loosely dotted, red]   table[x index=0, y index =1]{IBSRNH32.dat};
        \addplot [line width=3pt, densely dash dot, yellow] table[x index=0,y index =1]{IBSRNH33.dat};
        \addplot [line width=3pt, densely dash dot, green] table[x index=0,y index =1]{IBSRNH93.dat};
        \addlegendentry{ROBOBOA-D, AG}
        \addlegendentry{ROBOBOA-DF}
        \addlegendentry{HANSO-FI, AG, $\Ntheta^2=3^2$}
        \addlegendentry{HANSO-FI, AG, $\Ntheta^2=9^2$}
\end{axis}
\end{tikzpicture}
\end{minipage}
\hfill
\begin{minipage}[b]{0.475\textwidth}
\centering
\begin{tikzpicture}[baseline, scale=0.695]
    \begin{axis}[xlabel = {number of function evaluations},
                         ylabel = {$\rho_{\Psi_*}$},
                         legend pos = north east,
                        ymin=0.58, ymax=1,
                        xmin =0, xmax =2000,
                        legend cell align={left}
                             ]
	\draw[line width=1pt, loosely dotted, black] (0,0.6) -- (4000,0.6);
        \addplot  [line width=3pt, solid, blue]  table[x index=0, y index =1]{IBSRNH3-121.dat};
        \addplot  [line width=3pt, loosely dotted, red]   table[x index=0, y index =1]{IBSRNH3-122.dat};
        \addplot [line width=3pt, densely dash dot, yellow] table[x index=0,y index =1]{IBSRNH3-123.dat};
        \addplot [line width=3pt, densely dash dot, green] table[x index=0,y index =1]{IBSRNH9-123.dat};
        \addlegendentry{ROBOBOA-D, FDG-$t$=1e-12}
        \addlegendentry{ROBOBOA-DF}
        \addlegendentry{HANSO-FI, FDG-$t$=1e-12, $\Ntheta^2=3^2$}
        \addlegendentry{HANSO-FI, FDG-$t$=1e-12, $\Ntheta^2=9^2$}
\end{axis}
\end{tikzpicture}
\end{minipage}
\caption{Optimization results for the MAC scheme discretization with inexact Braess-Sarazin relaxation \cref{Precondtion}. The derivatives are obtained analytically (left) or using central differences with  stepsize $10^{-12}$ (right). } \label{odd-compare-IBSR-MAC}
\end{figure}

The Uzawa relaxation scheme can be viewed as a block triangular approximation of Braess-Sarazin, solving
\begin{equation}\label{Precondtion-Uzawa}
   M_U \delta x=  \begin{pmatrix}
      p_2 D & 0\\
     B & p_3^{-1} I\\
    \end{pmatrix}
    \begin{pmatrix} \delta \mathcal{U} \\ \delta {\rm p}\end{pmatrix}
  =\begin{pmatrix} r_{\mathcal{U}} \\ r_{\rm p}\end{pmatrix}.
\end{equation}
We again consider a three-dimensional parameter space, $(p_1,p_2,p_3)$, with outer damping parameter $p_1$.

From \cite{NLA2147}, we know that the optimal LFA two-grid convergence factor is
\begin{equation*}
\rho_{\rm opt} = \min_{\pb \in \Rp^3}\max_{\thetab \in[-\frac{\pi}{2},\frac{\pi}{2}]^2}\Big\{\rho\big(\widetilde{E}_{U}((p_1,p_2,p_3), \thetab)\big)\Big\}
  =\sqrt{\frac{3}{5}}\approx 0.775,
\end{equation*}
with  multiple solutions of the minimax problem, including  $(p_1,p_2,p_3)=(1,\frac{5}{4},\frac{1}{4})$.

\Cref{Odd-compare-Uzawa-MAC} again compares the use of analytical calculations and central-difference approximations for the derivatives.  In both cases, \HANSO\ (with the coarse sampling $\Ntheta=3$) quickly approaches the optimal convergence factor;
\ROBOBOA\ without derivatives finds $\rho_{\Psi_*} \leq 0.775$ in roughly 800 function evaluations. When using analytical derivatives, \ROBOBOA\ with derivatives performs similarly to without derivatives but is less efficient in the central-difference case.
This provides an example where adaptive sampling with less accurate derivatives can require additional evaluations.
While this is almost five times slower than \ROBOBOA\ without derivatives, it still represents a great improvement on even a coarsely discretized brute-force approach, which might require $10^5$ (or more) function evaluations to sample on a mesh in three parameter directions and two Fourier frequencies.

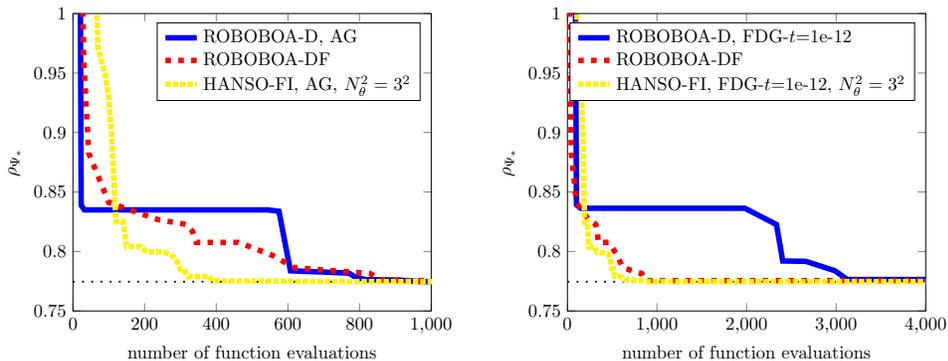
\begin{figure}
\centering
\begin{minipage}[b]{0.495\textwidth}
\centering
\begin{tikzpicture}[baseline, scale=0.695]
    \begin{axis}[xlabel = {number of function evaluations},
                         ylabel = {$\rho_{\Psi_*}$},
                         legend pos = north east,
                        ymin=0.75, ymax=1,
                        xmin =0, xmax =1000,
                        legend cell align={left}
                             ]
	\draw[line width=1pt, loosely dotted, black] (0,0.7746) -- (4000,0.7746);
        \addplot  [line width=3pt, solid, blue]  table[x index=0, y index =1]{MAC-UzawaNH31.dat};
        \addplot  [line width=3pt, loosely dotted, red]   table[x index=0, y index =1]{MAC-UzawaNH32.dat};
        \addplot [line width=3pt, densely dash dot, yellow] table[x index=0,y index =1]{MAC-UzawaNH33.dat};
        \addlegendentry{ROBOBOA-D, AG}
        \addlegendentry{ROBOBOA-DF}
        \addlegendentry{HANSO-FI, AG, $\Ntheta^2=3^2$}
\end{axis}
\end{tikzpicture}
\end{minipage}
\hfill
\begin{minipage}[b]{0.495\textwidth}
\centering
\begin{tikzpicture}[baseline, scale=0.695]
    \begin{axis}[xlabel = {number of function evaluations},
                         ylabel = {$\rho_{\Psi_*}$},
                         legend pos = north east,
                        ymin=0.75, ymax=1,
                        xmin =0, xmax =4000,
                        legend cell align={left}
                             ]
	\draw[line width=1pt, loosely dotted, black] (0,0.7746) -- (4000,0.7746);
        \addplot  [line width=3pt, solid, blue]  table[x index=0, y index =1]{Uzawa4000NH3-121.dat};
        \addplot  [line width=3pt, loosely dotted, red]   table[x index=0, y index =1]{Uzawa4000NH3-122.dat};
        \addplot [line width=3pt, densely dash dot, yellow] table[x index=0,y index =1]{Uzawa4000NH3-123.dat};
        \addlegendentry{ROBOBOA-D, FDG-$t$=1e-12}
        \addlegendentry{ROBOBOA-DF}
        \addlegendentry{HANSO-FI, FDG-$t$=1e-12, $\Ntheta^2=3^2$}
\end{axis}
\end{tikzpicture}
\end{minipage}
\caption{Optimization performance for MAC finite-difference discretization using Uzawa relaxation \cref{Precondtion-Uzawa}. Left: results using analytical derivatives; right: results using central differences with stepsize  $10^{-12}$.
Note differences in horizontal axis limits.} \label{Odd-compare-Uzawa-MAC}
\end{figure}

\subsubsection{Stabilized $Q_1-Q_1$ Discretization}\label{subsec:Q1-Q1-Stokes}
We now consider the distributive weighted-Jacobi (DWJ) relaxation analyzed in \cite{HMFEMStokes}.  The idea of distributive relaxation is to replace  relaxation on the equation $Ly = b$ by introducing a new variable, $\hat{y}$ with $y= F\hat{y}$,
and considering relaxation on the transformed system $Ly = LF\hat{y} = b$.
Here, $F$ is chosen such that the resulting operator $LF$ is suitable for decoupled relaxation with a
pointwise relaxation process.  For Stokes, it is common to take
\begin{equation*}\label{P-Precondtion}
   F=  \begin{pmatrix}
      I & B^{T}\\
      0 & -A_p\\
    \end{pmatrix},
\end{equation*}
where $A_p$ is the Laplacian operator discretized at the pressure points.
 Here, we use
\begin{equation}\label{DWJ-Precondition-2Sweeps}
   M_{{\rm DWJ}}\delta x=  \begin{pmatrix}
      p_2 D & 0\\
      B & G\\
    \end{pmatrix}
    \begin{pmatrix} \delta \mathcal{{U}} \\ \delta {\rm{p}}\end{pmatrix}
  =\begin{pmatrix} r_{\mathcal{U}} \\ r_{\rm p}\end{pmatrix},
\end{equation}
where $G$ stands for applying either a scaling operation (equivalent to one sweep of weighted Jacobi relaxation) or two sweeps of weighted-Jacobi relaxation with equal weights, $p_3$, on the pressure equation in \cref{solution-of-precondtion}. An outer damping parameter, $p_1$, is used for the relaxation scheme so that $\mathcal{S}_h = I -p_1 F M_{\rm DWJ}^{-1}L_h$. This gives a three-dimensional parameter space
and $3\times 3$ symbols for each harmonic frequency, leading to a  $12\times 12$ symbol for the two-grid error propagation operator.

As found in \cite{HMFEMStokes}, the optimal convergence factor for a single sweep of relaxation on the pressure \revise{when using a rediscretization coarse-grid operator} is
$  \rho_{\rm opt} = \frac{55}{89}\approx 0.618,$
which is achieved if and only if
\begin{equation*}\label{beta-DWJ-Parameter-domain}
   \frac{p_1}{p_3}=\frac{459}{356},\,\,\frac{136}{267}\leq\frac{p_1}{p_2}\leq\frac{96}{89}.
\end{equation*}
\Cref{odd-compare-DWJ-FEM} shows the performance of the methods, using central-difference approximations of the derivatives for two different values of $t$. We see that the \HANSO\ variants and \ROBOBOA\ without derivatives are most effective, obtaining rapid decrease in roughly 1,000 function evaluations.
For both stepsizes $t$, \ROBOBOA\ with central differences suffers
(failing to
attain a similar
value within 10,000 function evaluations).

\revise{
  \begin{remark}
    In contrast to the results above, when a Galerkin coarse-grid operator is used, the corresponding two-grid LFA convergence factor degrades to $0.906$.  This is not particularly surprising, since the stabilization term depends on the meshsize, $h$, and the Galerkin coarse-grid operator does not properly ``rescale'' this value.  In this setting, we can again add a weighting factor to the coarse-grid correction process to attempt to recover improved performance.  By brute-force search, we find the optimal parameter is a damping factor of $0.70$, which recovers the two-grid convergence factor from the rediscretization case.  A more systematic exploration of the potential use of such factors in monolithic methods for the Stokes equations (and other coupled systems) is left for future work.
\end{remark}
}

\begin{figure}
\begin{minipage}{.48\textwidth}
  \centering
\begin{tikzpicture}[baseline, scale=0.73]
    \begin{axis}[xlabel = {number of function evaluations},
                         ylabel = {$\rho_{\Psi_*}$},
                         legend pos = north east,
                        ymin=0.6, ymax=1,
                        xmin =0, xmax =4000,
                        legend cell align={left}
                             ]
	\draw[line width=1pt, loosely dotted, black] (0,0.618) -- (4000,0.618);
        \addplot  [line width=3pt, densely dash dot, purple]  table[x index=0, y index =1]{DWJNH3-121.dat};
        \addplot  [line width=3pt, solid, blue]  table[x index=0, y index =1]{DWJNH3-61.dat};
        \addplot  [line width=3pt, loosely dotted, red]   table[x index=0, y index =1]{DWJNH3-62.dat};
        \addplot [line width=3pt, solid, green] table[x index=0,y index =1]{DWJNH3-123.dat};
        \addplot [line width=3pt, densely dash dot, yellow] table[x index=0,y index =1]{DWJNH3-63.dat};
        \addlegendentry{ROBOBOA-D, FDG-$t$=1e-12}
        \addlegendentry{ROBOBOA-D, FDG-$t$=1e-6}
        \addlegendentry{ROBOBOA-DF}
        \addlegendentry{HANSO-FI, FDG-$t$=1e-12, $\Ntheta^2=3^2$}
        \addlegendentry{HANSO-FI, FDG-$t$=1e-6, $\Ntheta^2=3^2$}
\end{axis}
\end{tikzpicture}
  \captionof{figure}{Optimization performance for the stabilized $Q_1-Q_1$ discretization of the Stokes equations, using distributive weighted Jacobi relaxation \cref{DWJ-Precondition-2Sweeps} with a single sweep of relaxation on the transformed pressure equation.
}
\label{odd-compare-DWJ-FEM}
\end{minipage}%
\hfill
\begin{minipage}{.48\textwidth}
  \centering
\begin{tikzpicture}[baseline, scale=0.73]
    \begin{axis}[xlabel = {number of function evaluations},
                         ylabel = {$\rho_{\Psi_*}$},
                         legend pos = north east,
                        ymin=0.31, ymax=1,
                        xmin =0, xmax =4000,
                        legend cell align={left}
                             ]
	\draw[line width=1pt, loosely dotted, black] (0,0.333) -- (4000,0.333);
        \addplot  [line width=3pt, solid, blue]  table[x index=0, y index =1]{DWJ2NH3-61.dat};
        \addplot  [line width=3pt, loosely dotted, red]   table[x index=0, y index =1]{DWJ2NH3-62.dat};
        \addplot [line width=3pt, densely dash dot, yellow] table[x index=0,y index =1]{DWJ2NH3-63.dat};
         \addplot [line width=3pt, densely dash dot, green] table[x index=0,y index =1]{DWJ2NH8-63.dat};
        \addplot [line width=3pt, densely dash dot, orange] table[x index=0,y index =1]{DWJ2NH9-63.dat};
        \addlegendentry{ROBOBOA-D, FDG-$t$=1e-6}
        \addlegendentry{ROBOBOA-DF}
        \addlegendentry{HANSO-FI, FDG-$t$=1e-6, $\Ntheta^2=3^2$}
        \addlegendentry{HANSO-FI, FDG-$t$=1e-6, $\Ntheta^2=8^2$}
        \addlegendentry{HANSO-FI, FDG-$t$=1e-6, $\Ntheta^2=9^2$}
\end{axis}
\end{tikzpicture}
  \captionof{figure}{Optimization performance for the stabilized $Q_1-Q_1$ discretization of the Stokes equations, using distributive weighted Jacobi relaxation \cref{DWJ-Precondition-2Sweeps} with two sweeps of relaxation on the transformed pressure equation.
} \label{odd-compare-DWJ2}
\end{minipage}%
\end{figure}

As shown in  \cite{HMFEMStokes}, using two sweeps of weighted Jacobi on the transformed pressure equation greatly improves the performance of the DWJ relaxation, yielding $\rho_{\rm opt}=\frac{1}{3}$ with many solutions, including $(p_1, p_2,p_3)=(\frac{4}{3},  \frac{3}{2}, 1)$, achieving this value.
In \Cref{odd-compare-DWJ2}, we use central differences to evaluate the derivatives; shown are the results with $t=10^{-6}$, the test stepsize for which both methods performed best.
In contrast to the single-sweep case, \ROBOBOA\ with derivatives performs best and outperforms \ROBOBOA\ without derivatives.
This is also a case where sampling at $\thetab=0$ is not vital;
although the initial decrease obtained by \HANSO\ is as expected and ordered in terms of the sampling rate $\Ntheta$, only the $\Ntheta^2=8^2$ variant is close to approaching the $\rho_{\rm opt}$ value. The $\Ntheta^2=3^2$ and $\Ntheta^2=9^2$ variants of \HANSO\ converge (up to 10,000 function evaluations were tested) to parameter values with $\rho_{\Psi_*} \approx 0.4$.
We note that a brute-force discretized search with $N_p=20$ and $\Ntheta=33$ costs $20^3\cdot 33^2\approx 8.7\times 10^{6}$ function evaluations, considerably more expensive than any of these approaches.

\subsubsection{Additive Vanka  Relaxation for Stokes}\label{sec:Vanka-case}
We next consider the stable $P_2-P_1$ approximation for the Stokes equations with  an additive Vanka relaxation scheme \cite{StokePatchFHM}. \revise{
  Vanka relaxation, a form of overlapping Schwarz iteration, is well known as an effective relaxation scheme for the Stokes equations when used in its multiplicative form \cite{SPMacLachlan_CWOosterlee_2011a,vanka1986block}.  However, as is typical, the multiplicative scheme is less readily parallelized than its additive counterpart.  This has driven recent interest in additive Vanka relaxation schemes.  The additive form of Vanka relaxation, however, is more sensitive to the choice of relaxation parameters, which is the subject of \cite{StokePatchFHM} and the problem tackled here.  This sensitivity makes the problem attractive as a test case for the methods proposed here.}

We first consider the case of $\nu_1+\nu_2=2$ sweeps of additive Vanka-inclusive relaxation that uses $n=3$ relaxation weights, separately weighting the corrections to nodal and edge velocity degrees  of freedom and  the (nodal) pressure degrees of freedom.  Because of the structure of the $P_2-P_1$ discretization, the symbol of each harmonic frequency is a $9\times 9$ matrix; consequently, that of the two-grid operator is a $36\times 36$ matrix.
A brute-force discretized search using $N_p=15$ points for each parameter and $\Ntheta=33$ points for each frequency finds parameters that yield
 $\rho_{\Psi_*} \approx 0.455$ using 3,675,375  function evaluations.   \Cref{odd-compare-VK3} (left) shows
that
\ROBOBOA\ without derivatives attains this value in roughly 3,400 function evaluations.
For the \ROBOBOA\ and \HANSO\ variants with approximate derivatives, rapid decrease is seen, but the values found have $\rho_{\Psi_*}\in [0.45,0.5]$, independent of the central-difference stepsize.

An even more challenging problem is shown in \Cref{odd-compare-VK3} (right), where we optimize $n=5$ relaxation parameters by including separate relaxation parameters for each type of $P_2$ degree of freedom in the velocity discretization. Here, a brute-force discretized search is out of scope because of the high dimensionality of the problem.
Only through advanced optimization can one analyze the benefit of the additional parameters.
In this case, the multigrid performance is not improved by increasing the degrees of freedom from $n=3$ to $n=5$.

\begin{figure}
\centering
\begin{minipage}[b]{0.485\textwidth}
\centering
\begin{tikzpicture}[baseline, scale=0.695]
    \begin{axis}[xlabel = {number of function evaluations},
                         ylabel = {$\rho_{\Psi_*}$},
                         legend pos = north east,
                        ymin=0.4, ymax=0.9,
                        xmin =0, xmax =5000,
                        legend cell align={left}
                             ]
        \addplot  [line width=3pt, densely dash dot, purple]  table[x index=0, y index =1]{VK3NH3-121.dat};
        \addplot  [line width=3pt, solid, blue]  table[x index=0, y index =1]{VK3NH3-61.dat};
        \addplot  [line width=3pt, loosely dotted, red]   table[x index=0, y index =1]{VK3NH3-122.dat};
        \addplot [line width=3pt, solid, green] table[x index=0,y index =1]{VK3NH3-123.dat};
        \addplot [line width=3pt, densely dash dot, yellow] table[x index=0,y index =1]{VK3NH3-63.dat};
        \addlegendentry{ROBOBOA-D, FDG-$t$=1e-12}
        \addlegendentry{ROBOBOA-D, FDG-$t$=1e-6}
        \addlegendentry{ROBOBOA-DF}
        \addlegendentry{HANSO-FI, FDG-$t$=1e-12, $\Ntheta^2=3^2$}
        \addlegendentry{HANSO-FI, FDG-$t$=1e-6, $\Ntheta^2=3^2$}
\end{axis}
\end{tikzpicture}
\end{minipage}
\hfill
\begin{minipage}[b]{0.485\textwidth}
\centering
\begin{tikzpicture}[baseline, scale=0.695]
    \begin{axis}[xlabel = {number of function evaluations},
                         ylabel = {$\rho_{\Psi_*}$},
                         legend pos = north east,
                        ymin=0.4, ymax=0.9,
                        xmin =0, xmax =5000,
                        legend cell align={left}
                             ]
        \addplot  [line width=3pt, densely dash dot, purple]  table[x index=0, y index =1]{VK5NH3-121.dat};
        \addplot  [line width=3pt, solid, blue]  table[x index=0, y index =1]{VK5NH3-61.dat};
        \addplot  [line width=3pt, loosely dotted, red]   table[x index=0, y index =1]{VK5NH3-122.dat};
        \addplot [line width=3pt, solid, green] table[x index=0,y index =1]{VK5NH3-123.dat};
        \addplot [line width=3pt, densely dash dot, yellow] table[x index=0,y index =1]{VK5NH3-63.dat};
        \addlegendentry{ROBOBOA-D, FDG-$t$=1e-12}
        \addlegendentry{ROBOBOA-D, FDG-$t$=1e-6}
        \addlegendentry{ROBOBOA-DF}
        \addlegendentry{HANSO-FI, FDG-$t$=1e-12, $\Ntheta^2=3^2$}
        \addlegendentry{HANSO-FI, FDG-$t$=1e-6, $\Ntheta^2=3^2$}
\end{axis}
\end{tikzpicture}
\end{minipage}
\caption{Optimization results for additive Vanka relaxation applied to the $P_2-P_1$ discretization of the Stokes equations using $n=3$ (left) and $n=5$ (right) parameters. Both cases use central differences with stepsize $10^{-6}$ and $10^{-12}$ to approximate derivatives.}\label{odd-compare-VK3}
\end{figure}
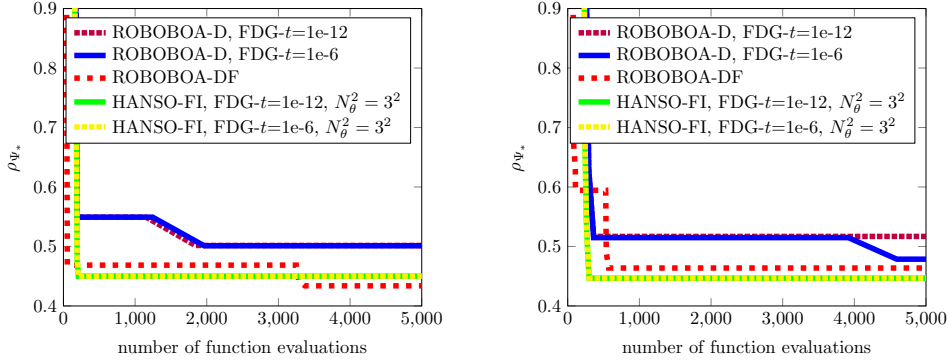

\subsection{Control Problem in 3D}\label{sec:Control3D}
Our final problem is a 3D elliptic optimal control
problem from \cite{kepler2009fourier} that seeks the state $z \in H^{1}(\Omega)$ and control $g\in L^{2}(\Omega)$ that solve
\begin{equation}\label{Control-problem}
  \begin{aligned}
&\displaystyle\min_{(z,g)\in H^{1}(\Omega)\times L^{2}(\Omega)} &
 \frac{1}{2}\|z-y_d\|^2_{L^2(\Omega)}+\frac{\beta}{2}\|g\|^2_{L^2(\Omega)} \\
& \mbox{subject to} &-\Delta z =g \qquad {\rm in} \; \Omega\\
&  &z=0 \qquad {\rm on} \; \partial\Omega,\\
   \end{aligned}
\end{equation}
where $y_d$ is the desired state and $\beta>0$ is the weight of the cost of the control
(or simply a regularization parameter).
We consider discretization using $Q_1$ finite elements, which yields
\begin{eqnarray*}
&&\min_{z_h,g_h}\frac{1}{2}\|z_h-y_d\|^2_{L^2(\Omega)}+\frac{\beta}{2}\|g_h\|^2_{L^2(\Omega)},\label{weak-form1}\\
&&\mbox{subject to: } \int_{\Omega} \nabla z_h\cdot \nabla v_h = \int_{\Omega} g_h v_h, \,\,\forall v_h\in V_0^h. \label{weak-form2}
\end{eqnarray*}
Using a Lagrange multiplier approach to enforce the constraint leads to a linear system of the form
\begin{equation*}\label{saddle-structure}
     \mathcal{A} x=\begin{pmatrix}
      M   & K^{T}   &0 \\
      K   & 0      & -M \\
      0   &-M^{T}  &\beta M
    \end{pmatrix}
        \begin{pmatrix} z_h \\ \lambda_h \\g_h \end{pmatrix}
  =\begin{pmatrix} f_{1} \\f_{2} \\ f_{3} \end{pmatrix}=b,
 \end{equation*}
 where $M$  and $K$ are the mass and stiffness matrices, respectively, of the $Q_1$ discretization of the Laplacian in 3D.
 We consider a block Jacobi relaxation scheme with
 \begin{equation*}\label{precond-A2}
     \hat{\mathcal{A}}=\begin{pmatrix}
      {\rm diag}(M)                &p_2{\rm diag}(K)     &0 \\
 p_2 {\rm diag}(K)  & 0                      &-{\rm diag}(M)\\
    0                  &-{\rm diag}(M^{T})                  & \beta {\rm diag}(M)
    \end{pmatrix}
\end{equation*}
 and a multigrid method that uses an outer damping parameter, $p_1$,  for the relaxation scheme (i.e., $\mathcal{S}_h = I -p_1  \hat{\mathcal{A}}^{-1}\mathcal{A}$). This gives an ($n=2$)-dimensional parameter space and $3\times 3$ symbols for each harmonic frequency, leading to a  $24\times 24$ symbol for the two-grid error propagation operator. \revise{To our knowledge, the optimization of these relaxation parameters has not been considered before; however, standard (if tedious) calculations show that the LFA smoothing factor achieves a minimum value of $\frac{17}{19}$ if and only if $p_1=\frac{16}{19}$ and $p_2\in[\frac{2}{3},4]$, which we have confirmed with numerical experiments.}

The results in \Cref{odd-compare-control-problem}  use central differences with stepsize $10^{-12}$ to approximate the derivatives.
The three tested methods find near-optimal values within 1,300 function evaluations, with the coarse-sampling ($\Ntheta=3$) \HANSO\ showing rapid improvement.

\begin{figure}
\begin{minipage}{.495\textwidth}
  \centering
\begin{tikzpicture}[baseline, scale=0.73]
    \begin{axis}[xlabel = {number of function evaluations},
                         ylabel = {$\rho_{\Psi_*}$},
                         legend pos = north east,
                        ymin=0.89, ymax=0.94,
                        xmin =0, xmax =2000,
                        legend cell align={left}
                             ]
 \draw[line width=1pt, loosely dotted, black] (0,0.895) -- (2000,0.895);
        \addplot  [line width=3pt, solid, blue]  table[x index=0, y index =1]{GCPNH3-121.dat};
        \addplot  [line width=3pt, loosely dotted, red]   table[x index=0, y index =1]{GCPNH3-122.dat};
        \addplot [line width=3pt, densely dash dot, yellow] table[x index=0,y index =1]{GCPNH3-123.dat};
        \addlegendentry{ROBOBOA-D, FDG-$t$=1e-12}
        \addlegendentry{ROBOBOA-DF}
        \addlegendentry{HANSO-FI, FDG-$t$=1e-12, $\Ntheta^3=3^3$}
\end{axis}
\end{tikzpicture}
  \captionof{figure}{Optimization performance for 3D control problem \cref{Control-problem}. \vspace{4pc}
}
\label{odd-compare-control-problem}
\end{minipage}%
\hfill
\begin{minipage}{.48\textwidth}
  \centering
\includegraphics[width=\linewidth]{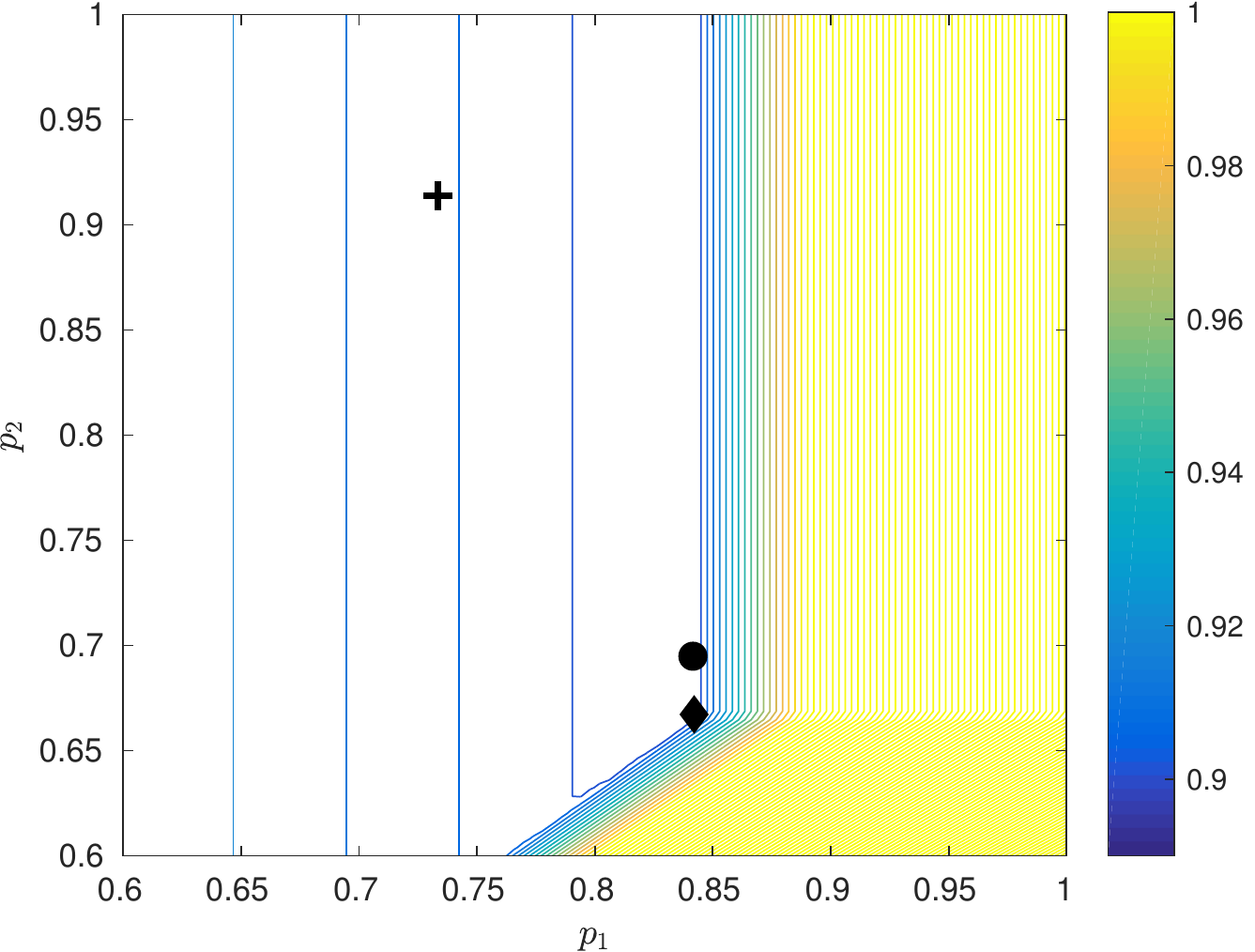}
   \captionof{figure}{$\Psi_{\thetab(\Ntheta^3=33^3)}(p_1,p_2)$ as a function of $p_1$ and $p_2$ for 3D control problem.
   Cross denotes solution found by ROBOBOA-D,
   circle denotes solution found by
   ROBOBOA-DF, and
   diamond denotes solution found by HANSO.}\label{GCP-contour-plt}
\end{minipage}%
\end{figure}

To explore the sensitivity of LFA-predicted convergence factors to parameter choice, in \Cref{GCP-contour-plt} we examine
$\Psi_{\thetab(\Ntheta^3=33^3)}(p_1,p_2)$ as a function of $p_1$ and $p_2$.
We see that the optimal parameters match with the analytical solutions where $p_1=\frac{16}{19}$ and $p_2\in[\frac{2}{3},4]$.

\section{Conclusions}\label{Conclusion}
While both analytical calculations and brute-force search have been used for many years to optimize multigrid parameter choices with LFA,
increasingly many problems are encountered where both these methods are computationally infeasible.
Here, we propose variants of recent robust optimization algorithms applied to the LFA minimax problem.
Numerical results show that these algorithms are, in general, capable of finding optimal (or near-optimal) parameters in a number of function evaluations orders of magnitude smaller
than brute-force search.
When analytical derivatives are available, \ROBOBOA\ with derivatives is often the most efficient approach, but its performance clearly suffers when using central-difference approximations.
\ROBOBOA\ without derivatives and \HANSO\ coupled with a brute-force inner search are less efficient but frequently successful.
In terms of consistently finding approximately optimal solutions, however, \ROBOBOA\ without derivatives seems to be the preferred method.
Moreover, a strong case can be made for \ROBOBOA\ in that \ROBOBOA\ does not depend heavily on an a priori discretization $\Ntheta^d$ but instead adapts to the most ``active'' Fourier modes (in the sense of those that dominate the multigrid performance prediction)
at a set of parameter values $\pb$. 
Thus, \ROBOBOA\ is a preferable method when not much is known analytically about an LFA problem.

Given the importance of accurate derivatives to the optimization process,
an obvious avenue for future work is to use automated differentiation tools in combination with ROBOBOA.
This poses more of a software engineering challenge than a conceptual one but may be able to leverage recent LFA software projects such as
\cite{RitExtending2017,kahl2018automated}. Similarly, better inner search approaches for \HANSO\ could lead to a better tool.
Important future work also lies in applying these optimization tools to design and improve multigrid algorithms for the complex systems of PDEs
that are of interest in modern computational science and engineering.  \revise{While we have focused primarily on the choice of relaxation parameters in this work, the tools developed could be directly applied to any part of the multigrid algorithm, including determining coefficients in the grid-transfer operators or scalings of stabilization terms in the coarse-grid equations.  The tools could also be applied to other families of preconditioners, such as determining optimized boundary conditions or two-level processes within Schwarz methods.}


\section*{Acknowledgments}

 S.M. and Y.H. thank Professor Michael Overton for discussing the use of HANSO for this optimization problem during his visit to Memorial University in 2018.  We thank the referees for helpful comments that improved the presentation in this paper.

\bibliographystyle{siam}
\bibliography{LFA_Optimization}

\begin{thebibliography}{10}

\bibitem{bolten2018fourier}
{\sc M.~Bolten and H.~Rittich}, {\em Fourier analysis of periodic stencils in
  multigrid methods}, SIAM Journal on Scientific Computing, 40 (2018),
  pp.~A1642--A1668.

\bibitem{brandt1977multi}
{\sc A.~Brandt}, {\em Multi-level adaptive solutions to boundary-value
  problems}, Mathematics of Computation, 31 (1977), pp.~333--390.

\bibitem{MR558216}
{\sc A.~Brandt and N.~Dinar}, {\em Multigrid solutions to elliptic flow
  problems}, in Numerical {M}ethods for {P}artial {D}ifferential {E}quations,
  vol.~42, Academic Press, 1979, pp.~53--147.

\bibitem{briggs2000multigrid}
{\sc W.~Briggs, V.~Henson, and S.~McCormick}, {\em A Multigrid Tutorial}, SIAM,
  2000.

\bibitem{JedBDDC}
{\sc J.~Brown, Y.~He, and S.~MacLachlan}, {\em Local {F}ourier analysis of
  {B}alancing {D}omain {D}ecomposition by {C}onstraints algorithms}, SIAM
  Journal on Scientific Computing, 41 (2019), pp.~S346--S369.

\bibitem{burke2005robust}
{\sc J.~Burke, A.~Lewis, and M.~Overton}, {\em A robust gradient sampling
  algorithm for nonsmooth, nonconvex optimization}, SIAM Journal on
  Optimization, 15 (2005), pp.~751--779.

\bibitem{MR1061136}
{\sc K.~Chu}, {\em On multiple eigenvalues of matrices depending on several
  parameters}, SIAM Journal on Numerical Analysis, 27 (1990), pp.~1368--1385.

\bibitem{Conn2009a}
{\sc A.~Conn, K.~Scheinberg, and L.~Vicente}, {\em Introduction to
  Derivative-Free Optimization}, SIAM, 2009.

\bibitem{drzisga2018analysis}
{\sc D.~Drzisga, L.~John, U.~R{\"u}de, B.~Wohlmuth, and W.~Zulehner}, {\em On
  the analysis of block smoothers for saddle point problems}, SIAM Journal on
  Matrix Analysis and Applications, 39 (2018), pp.~932--960.

\bibitem{elman2006finite}
{\sc H.~Elman, D.~Silvester, and A.~Wathen}, {\em Finite Elements and Fast
  Iterative Solvers with Applications in Incompressible Fluid Dynamics}, Oxford
  University Press, 2nd~ed., 2014.

\bibitem{StokePatchFHM}
{\sc P.~E. Farrell, Y.~He, and S.~P. MacLachlan}, {\em A local {Fourier}
  analysis of additive {Vanka}{ relaxation for the {Stokes} equations}},
  Numerical Linear Algebra with Applications, n/a (2020), p.~e2306.

\bibitem{franco2018multigrid}
{\sc S.~Franco, C.~Rodrigo, F.~Gaspar, and M.~Pinto}, {\em A multigrid waveform
  relaxation method for solving the poroelasticity equations}, Computational
  and Applied Mathematics,  (2018), pp.~1--16.

\bibitem{gaspar2014simple}
{\sc F.~Gaspar, Y.~Notay, C.~Oosterlee, and C.~Rodrigo}, {\em A simple and
  efficient segregated smoother for the discrete {S}tokes equations}, SIAM
  Journal on Scientific Computing, 36 (2014), pp.~A1187--A1206.

\bibitem{gaspar2009geometric}
{\sc F.~J. Gaspar, J.~L. Gracia, F.~J. Lisbona, and C.~Rodrigo}, {\em On
  geometric multigrid methods for triangular grids using three-coarsening
  strategy}, Applied Numerical Mathematics, 59 (2009), pp.~1693--1708.

\bibitem{grebhahn2014optimizing}
{\sc A.~Grebhahn, N.~Siegmund, S.~Apel, S.~Kuckuk, C.~Schmitt, and
  H.~K{\"o}stler}, {\em Optimizing performance of stencil code with {SPL}
  conqueror}, in Proceedings of the 1st International Workshop on
  High-Performance Stencil Computations (HiStencils), 2014, pp.~7--14.

\bibitem{Greenfeld2019Optimization}
{\sc D.~Greenfeld, M.~Galun, R.~Basri, and I.~Yavneh}, {\em Learning to
  optimize multigrid {PDE} solvers}, in International Conference on Machine
  Learning, 2019, pp.~2415--2423.

\bibitem{hackbusch2013multi}
{\sc W.~Hackbusch}, {\em Multi-Grid Methods and Applications}, vol.~4,
  Springer, 2013.

\bibitem{NLA2147}
{\sc Y.~He and S.~MacLachlan}, {\em Local {F}ourier analysis of
  block-structured multigrid relaxation schemes for the {S}tokes equations},
  Numerical Linear Algebra with Applications, 25 (2018).
\newblock e2147.

\bibitem{HMFEMStokes}
\leavevmode\vrule height 2pt depth -1.6pt width 23pt, {\em Local {F}ourier
  analysis for mixed finite-element methods for the {S}tokes equations},
  Journal of Computational and Applied Mathematics, 357 (2019), pp.~161--183.

\bibitem{HM2018LFALaplace}
{\sc Y.~He and S.~MacLachlan}, {\em Two-level {F}ourier analysis of multigrid
  for higher-order finite-element discretizations of the {L}aplacian},
  Numerical Linear Algebra with Applications, 27 (2020), p.~e2285.

\bibitem{kahl2018automated}
{\sc K.~Kahl and N.~Kintscher}, {\em Automated local {F}ourier analysis
  (a{LFA})}, BIT Numerical Mathematics,  (2020).

\bibitem{Katrutsa2020}
{\sc A.~Katrutsa, T.~Daulbaev, and I.~Oseledets}, {\em Black-box learning of
  multigrid parameters}, Journal of Computational and Applied Mathematics, 368
  (2020), p.~112524.

\bibitem{KLW18}
{\sc K.~Khan, J.~Larson, and S.~Wild}, {\em Manifold sampling for optimization
  of nonconvex functions that are piecewise linear compositions of smooth
  components}, {SIAM} Journal on Optimization, 28 (2018), pp.~3001--3024.

\bibitem{kumar2019local}
{\sc P.~Kumar, C.~Rodrigo, F.~Gaspar, and C.~Oosterlee}, {\em On local
  {Fourier} analysis of multigrid methods for {PDEs} with jumping and random
  coefficients}, SIAM Journal on Scientific Computing, 41 (2019),
  pp.~A1385--A1413.

\bibitem{lancaster1964eigenvalues}
{\sc P.~Lancaster}, {\em On eigenvalues of matrices dependent on a parameter},
  Numerische Mathematik, 6 (1964), pp.~377--387.

\bibitem{LMW16}
{\sc J.~Larson, M.~Menickelly, and S.~Wild}, {\em Manifold sampling for
  $\ell_1$ nonconvex optimization}, SIAM Journal on Optimization, 26 (2016),
  pp.~2540--2563.

\bibitem{LMW2019AN}
\leavevmode\vrule height 2pt depth -1.6pt width 23pt, {\em Derivative-free
  optimization methods}, Acta Numerica, 28 (2019), pp.~287--404.

\bibitem{LewOve12}
{\sc A.~Lewis and M.~Overton}, {\em Nonsmooth optimization via quasi-{Newton}
  methods}, Mathematical Programming, 141 (2013), pp.~135--163.

\bibitem{luo2017uzawa}
{\sc P.~Luo, C.~Rodrigo, F.~Gaspar, and C.~Oosterlee}, {\em On an {U}zawa
  smoother in multigrid for poroelasticity equations}, Numerical Linear Algebra
  with Applications, 24 (2017), p.~e2074.

\bibitem{Luo_etal_2017a}
\leavevmode\vrule height 2pt depth -1.6pt width 23pt, {\em Uzawa smoother in
  multigrid for the coupled porous medium and {S}tokes flow system}, SIAM
  Journal on Scientific Computing, 39 (2017), pp.~S633--S661.

\bibitem{luo2018monolithic}
\leavevmode\vrule height 2pt depth -1.6pt width 23pt, {\em Monolithic multigrid
  method for the coupled {S}tokes flow and deformable porous medium system},
  Journal of Computational Physics, 353 (2018), pp.~148--168.

\bibitem{Luz2020LearningAMG}
{\sc I.~Luz, M.~Galun, H.~Maron, R.~Basri, and I.~Yavneh}, {\em Learning
  algebraic multigrid using graph neural networks}, in International Conference
  on Machine Learning, 2020.

\bibitem{SPMacLachlan_CWOosterlee_2011a}
{\sc S.~Mac{L}achlan and C.~Oosterlee}, {\em Local {F}ourier analysis for
  multigrid with overlapping smoothers applied to systems of {PDE}s}, Numerical
  Linear Algebra with Applications, 18 (2011), pp.~751--774.

\bibitem{magnus1985differentiating}
{\sc J.~Magnus}, {\em On differentiating eigenvalues and eigenvectors},
  Econometric Theory, 1 (1985), pp.~179--191.

\bibitem{menickelly2017derivative}
{\sc M.~Menickelly and S.~Wild}, {\em Derivative-free robust optimization by
  outer approximations}, Mathematical Programming, 179 (2020), pp.~157--193.

\bibitem{more2011edn}
{\sc J.~Mor{\'e} and S.~Wild}, {\em Estimating derivatives of noisy
  simulations}, ACM Transactions on Mathematical Software, 38 (2012),
  pp.~19:1--19:21.

\bibitem{niestegge1990analysis}
{\sc A.~Niestegge and K.~Witsch}, {\em Analysis of a multigrid {S}tokes
  solver}, Applied Mathematics and Computation, 35 (1990), pp.~291--303.

\bibitem{CWOosterlee_RWienands_2002a}
{\sc C.~Oosterlee and R.~Wienands}, {\em A genetic search for optimal multigrid
  components within a {F}ourier analysis setting}, SIAM Journal on Scientific
  Computing, 24 (2002), pp.~924--944.

\bibitem{pernice-walker-1998}
{\sc M.~Pernice and H.~Walker}, {\em {NITSOL}: A {Newton} iterative solver for
  nonlinear systems}, SIAM Journal on Scientific Computing, 19 (1998),
  pp.~302--318.

\bibitem{Polak1997}
{\sc E.~Polak}, {\em Optimization}, Springer, 1997.

\bibitem{RitExtending2017}
{\sc H.~Rittich}, {\em Extending and Automating {F}ourier Analysis for
  Multigrid Methods}, PhD thesis, University of Wuppertal, 2017.

\bibitem{ruge1987algebraic}
{\sc J.~Ruge and K.~St{\"u}ben}, {\em Algebraic multigrid}, in Multigrid
  Methods, SIAM, 1987, pp.~73--130.

\bibitem{schmitt2019optimizing}
{\sc J.~Schmitt, S.~Kuckuk, and H.~K\"{o}stler}, {\em Constructing efficient
  multigrid solvers with genetic programming}, in Proceedings of the 2020
  Genetic and Evolutionary Computation Conference, {ACM}, 2020.

\bibitem{sivaloganathan1991use}
{\sc S.~Sivaloganathan}, {\em The use of local mode analysis in the design and
  comparison of multigrid methods}, Computer Physics Communications, 65 (1991),
  pp.~246--252.

\bibitem{de2019optimizing}
{\sc H.~D. Sterck, R.~Falgout, S.~Friedhoff, O.~Krzysik, and S.~MacLachlan},
  {\em Optimizing {MGRIT} and parareal coarse-grid operators for linear
  advection}, arXiv preprint arXiv:1910.03726,  (2019).

\bibitem{stuben1982multigrid}
{\sc K.~St{\"u}ben and U.~Trottenberg}, {\em Multigrid methods: {F}undamental
  algorithms, model problem analysis and applications}, in Multigrid Methods,
  vol.~960, Springer, 1982, pp.~1--176.

\bibitem{MR1807961}
{\sc U.~Trottenberg, C.~Oosterlee, and A.~Sch{\"u}ller}, {\em Multigrid},
  Academic Press, 2000.

\bibitem{vanka1986block}
{\sc S.~Vanka}, {\em Block-implicit multigrid solution of {N}avier-{S}tokes
  equations in primitive variables}, Journal of Computational Physics, 65
  (1986), pp.~138--158.

\bibitem{MR1156079}
{\sc P.~Wesseling}, {\em An Introduction to Multigrid Methods}, Wiley, 1992.

\bibitem{wienands2004practical}
{\sc R.~Wienands and W.~Joppich}, {\em Practical {F}ourier Analysis for
  Multigrid Methods}, CRC Press, 2004.

\bibitem{kepler2009fourier}
{\sc Y.~Zhao}, {\em Fourier analysis and local {F}ourier analysis for multigrid
  methods}, Master's thesis, Johannes Kepler Universit{\"a}t, 2009.

\end{thebibliography}

\clearpage
\appendix
\section{Simple Brute-Force Discretized Search}

\begin{algorithm}[H]
\caption{Brute-Force Discretized Search}
\label{alg:brute-force}
\begin{algorithmic}[1]
\begin{scriptsize}
\Require{$N_p, \Ntheta$ (number of samples for parameter and frequency in each dimension)}
\Ensure{$\rho_{\rm opt}$ and $\pb_{\rm opt}$ (approximate minimax solution)}
\Statex
\Function{OutLoop}{$N_p,\Ntheta$}
\State {$p_k={\rm Linspace}(a_k, b_k,N_p)_{k=1}^n$ and $\theta_i={\rm Linspace}(-\frac{\pi}{2},\frac{\pi}{2},\Ntheta)_{i=1}^d$}
\State {$\rho_{\rm opt}=1$}

    \For{$k_1 \gets 1$ to $N_p$}
            \For{$\ddots \gets 1$ to $N_p$}
              \For{$k_n \gets 1$ to $N_p$}

                \State {$\pb\leftarrow\left(p_1({k_1}),p_2({k_2}),\cdots,p_n({k_n})\right)$\hfill \% sample $\pb$}
           \State {$\rho_I=0$ \hfill \% store approximation to $\max_{\boldsymbol{t}} \rho\big(\widetilde{E}(\pb,\boldsymbol{t})\big)$}
                 \For{$j_1 \gets 1$ to $\Ntheta$}
                         \For{$\ddots \gets 1$ to $\Ntheta$}
                           \For {$j_d \gets 1$ to $\Ntheta$}
                            \State {$\thetab=\big(\theta_1(j_1),\theta_2(j_2),\cdots,\theta_d(j_d)\big)$\hfill \% sample $\thetab$}
                            \State {Compute $\rho^{*}=\rho\big(\widetilde{E}(\pb,\thetab)\big)$ \hfill \% spectral radius of $\widetilde{E}$ at sampled $(\pb,\thetab)$}
                            \If {$\rho^{*}\geq \rho_I$}
                            \State {$\rho_{I} \gets \rho^{*}$ }
                            \EndIf

                           \EndFor

                         \EndFor
                  \EndFor

\vspace{1mm}
               \If{$\rho_{\rm opt}\geq \rho_I$}
                \State {$\rho_{\rm opt}\Leftarrow\rho_I$ and $\pb_{\rm opt}\Leftarrow\pb$ \hfill \% update minimax approximation}
                \EndIf
                \vspace{1mm}

              \EndFor
            \EndFor
    \EndFor
    \State \Return {$\rho_{\rm opt}$ and $\pb_{\rm opt}$}
\EndFunction
\end{scriptsize}
\end{algorithmic}
\end{algorithm}

\section{Obtained Parameter Values}
\label{sec:params}

Here we report the first three digits for the parameters $\pb$ obtained by each method in the plots provided,
the corresponding value of $\rho_{\Psi_*}(\pb)$, and an approximate measure of stationarity of $\Psi$ at $\pb$, $\sigma(\pb)$.
The measure $\sigma(\pb)$ is intended to approximate the quantity
\begin{equation}
 \left\|\textbf{Proj}\left(0 \, | \, \textbf{co}\left(\displaystyle\cup_{\yb\in\cB(\pb,10^{-3})}\partial_C\Psi(\yb)\right)\right)\right\|,
 \label{eq:sigmap}
\end{equation}
where $\cB(\pb,10^{-3})$ denotes the Euclidean norm ball of radius $10^{-3}$ centered at $\pb$,
$\partial_C \Psi(\yb)$ denotes the (convex-set-valued) Clarke subdifferential of $\Psi$ at $\yb$,
$\textbf{Proj}\left(0 \, | \, \cdot\right)$ denotes the (unique) projection of the zero vector onto a given (convex) set,
and $\textbf{co}$ denotes the (set-valued) convex hull operator.
\eqref{eq:sigmap} is a natural measure of $\epsilon$-stationarity for nonconvex, nonsmooth, but locally Lipschitz functions $\Psi$;
here, we have taken $\epsilon=10^{-3}$.
Computing \eqref{eq:sigmap} is impractical, however, so we replace the set $\cB(\pb,10^{-3})$ in \eqref{eq:sigmap}
with the proper subset $\tilde\cB(\pb,10^{-3}) = \left\{\pb\right\}\cup\left\{\pb \pm 10^{-3}\mathbf{e}_j\right\}_{j=1}^n$,
where $\mathbf{e}_j$ denotes the $j$th elementary basis vector.
Because $\Psi$ is almost everywhere differentiable, we can define almost everywhere the quantity
$$ \sigma(\pb) = \left\|\textbf{Proj}\left(0 \, | \, \textbf{co}\left(\cup_{\yb\in\tilde\cB(\pb,10^{-3})}\nabla\Psi(\yb)\right)\right)\right\|,$$
the computation of which entails the solution of a convex quadratic optimization problem.
In our experiments, no method returned parameter values $\pb$ for which $\nabla\Psi(\pb)$ did not exist.
We note that $\sigma(\pb)$ upper bounds the quantity in \eqref{eq:sigmap}.
Thus, a nearly zero value for $\sigma(\pb)$ implies that $\pb$ is nearly $10^{-3}$-stationary.

\revise{We remark that the figures corresponding to the tables below exhibit
the best-found value of $\rho_{\Psi_*}$ (see \cref{eq:rho_psi_metric}) within a given number of function evaluations. 
However, because none of the algorithms tested actually compute the value of $\rho_{\Psi_*}$ during the optimization,
the final reported values of $\pb$ (and, hence, $\rho_{\Psi_*}(\pb)$) may not be the best-obtained during the optimization. 
This accounts for a few instances of discrepancies between the figures and tables (see, for example,  \Cref{odd-compare-LP1-coarsen3}).}

\begin{table}[p]
\caption{Parameter values and properties \label{tab:params}}
\raggedright \Cref{Odd-compare-LP1}: Laplace in 1D with a single relaxation,  \cref{eq:Q1Lp}
\\
\begin{tabular}{|l||l|l|l|} \hline
\textbf{Method} & $p_1$ & $\mathbf{\rho_{\Psi_*}(\pb)}$ & $\mathbf{\sigma}(\pb)$ \\ \hline \hline
ROBOBOA-D, AG & 0.667 & 0.333 & 1.110e-16 \\
ROBOBOA-DF & 0.667 & 0.333 & 0\\
HANSO-FI, AG, $\Ntheta=3$ & 0.667 & 0.333 & 1.110e-16 \\
HANSO-FI, AG, $\Ntheta=33$ & 0.667 & 0.333 & 1.110e-16 \\
\hline
\end{tabular}
\\
\raggedright \Cref{Odd-compare-LP2-NH4}: Laplace in 1D using the $P_1$ discretization with two sweeps of Jacobi relaxation \cref{eq:E2Laplace}
\\
\begin{tabular}{|l||l|l|l|l|} \hline
\textbf{Method} & $p_1$ & $p_2$ & $\mathbf{\rho_{\Psi_*}(\pb)}$ & $\mathbf{\sigma}(\pb)$ \\ \hline \hline
ROBOBOA-D, AG  & 1.000 & 0.500 & 0.000 & 1.145e-16 \\
ROBOBOA-DF & 1.000 & 0.500 & 0.000 & 2.372e-16 \\
HANSO-FI, AG, $\Ntheta=3$ & 0.667 & 0.667 & 0.111 & 8.135e-15 \\
\hline
\end{tabular}
\\
\raggedright \Cref{odd-compare-LP3}: $P_2$ discretization of the Laplace equation in 2D
\\
\begin{tabular}{|l||l|l|l|l|} \hline
\textbf{Method} & $p_1$ & $\mathbf{\rho_{\Psi_*}(\pb)}$ & $\mathbf{\sigma}(\pb)$ \\ \hline \hline
ROBOBOA-D, FDG-$t$=1e-8 & 0.851 & 0.610 & 1.891 \\
ROBOBOA-DF & 0.835 & 0.589 & 0 \\
HANSO-FI, FDG-$t$=1e-8, $\Ntheta^2=3^2$ & 0.849 & 0.605 & 1.891 \\
\hline
\end{tabular}
\\
\raggedright \Cref{odd-compare-LP1-coarsen3}: $P_1$ discretization of the Laplace equation in 1D with coarsening-by-threes and piecewise constant interpolation
\\
\begin{tabular}{|l||l|l|l|l|} \hline
\textbf{Method} & $p_1$ & $p_2$ & $\mathbf{\rho_{\Psi_*}(\pb)}$ & $\mathbf{\sigma}(\pb)$ \\ \hline \hline
ROBOBOA-D, FDG-$t$=1e-8  & 0.671 & 2.479 & 0.442 & 0.145 \\
ROBOBOA-DF & 0.741 & 2.249 & 0.429 & 0 \\
HANSO-FI, FDG-$t$=1e-8 $\Ntheta=3$ & 0.775 & 2.073 & 0.461 & 0.847 \\
\hline
\end{tabular}
\\
\raggedright \Cref{odd-compare-IBSR-MAC}: MAC scheme discretization with inexact Braess-Sarazin relaxation \cref{Precondtion}
\\
\begin{tabular}{|l||l|l|l|l|l|} \hline
\textbf{Method} & $p_1$ & $p_2$ & $p_3$ & $\mathbf{\rho_{\Psi_*}(\pb)}$ & $\mathbf{\sigma}(\pb)$ \\ \hline \hline
ROBOBOA-D, AG  & 1.208 & 0.967 & 0.799 & 0.601 & 0.021 \\
ROBOBOA-DF & 1.137 & 0.910 & 0.793 & 0.604 & 0.086 \\
HANSO-FI, AG, $\Ntheta^2=3^2$ & 1.250 & 1.000 & 0.800 & 0.600 & 2.074e-13 \\
HANSO-FI, AG, $\Ntheta^2=9^2$ & 1.250 & 1.000 & 0.800 & 0.600 & 1.634e-4 \\
ROBOBOA-D, FDG-$t$=1e-12 & 1.209 & 0.963 & 0.805 & 0.603 & 1.398 \\
HANSO-FI, FDG-$t$=1e-12, $\Ntheta^2=3^2$ & 1.249 & 1.000 & 0.800 & 0.600 & 3.480e-4 \\
HANSO-FI, FDG-$t$=1e-12, $\Ntheta^2=9^2$ & 1.224 & 0.979 & 0.800 & 0.600 & 0.018 \\
\hline
\end{tabular}
\\
\raggedright \Cref{Odd-compare-Uzawa-MAC}: MAC finite-difference discretization using Uzawa relaxation \cref{Precondtion-Uzawa}
\\
\begin{tabular}{|l||l|l|l|l|l|} \hline
\textbf{Method} & $p_1$ & $p_2$ & $p_3$ & $\mathbf{\rho_{\Psi_*}(\pb)}$ & $\mathbf{\sigma}(\pb)$ \\ \hline \hline
ROBOBOA-D, AG  & 1.918 & 1.691 & 0.134 & 0.775 & 5.176e-5 \\
ROBOBOA-DF & 0.914 & 0.618 & 0.479 & 0.775 & 3.500e-4 \\
HANSO-FI, AG, $\Ntheta^2=3^2$ & 1.191 & 0.937 & 0.272 & 0.775 & 6.317e-5 \\
ROBOBOA-D, FDG-$t$=1e-12  & 1.953 & 1.725 & 0.135 & 0.775 & 0.268 \\
HANSO-FI, FDG-$t$=1e-12, $\Ntheta^2=3^2$ & 1.197 & 0.943 & 0.269 & 0.775 & 3.187e-13 \\
\hline
\end{tabular}
\\
\raggedright \Cref{odd-compare-DWJ-FEM}: $Q_1-Q_1$ discretization of the Stokes equations, using DWJ relaxation \cref{DWJ-Precondition-2Sweeps} with a single sweep of relaxation on the transformed pressure equation
\\
\begin{tabular}{|l||l|l|l|l|l|} \hline
\textbf{Method} & $p_1$ & $p_2$ & $p_3$ & $\mathbf{\rho_{\Psi_*}(\pb)}$ & $\mathbf{\sigma}(\pb)$ \\ \hline \hline
ROBOBOA-D, FDG-$t$=1e-12 & 0.574 & 0.760 & 0.666 & 0.813 & 2.196 \\
ROBOBOA-D, FDG-$t$=1e-6  & 2.297 & 1.279 & 1.650 & 0.618 & 8.975e-5 \\
ROBOBOA-DF & 3.080 & 1.553 & 2.003 & 0.618 & 0 \\
HANSO-FI, FDG-$t$=1e-12, $\Ntheta^2=3^2$ & 3.046 & 1.698 & 2.292 & 0.694 & 1.239 \\
HANSO-FI, FDG-$t$=1e-6, $\Ntheta^2=3^2$ & 2.964 & 1.670 & 2.255 & 0.694 & 1.261 \\
\hline
\end{tabular}
\end{table}

\begin{table}[p]
\caption{Additional parameter values and properties \label{tab:params2}}
\raggedright \Cref{odd-compare-DWJ2}: $Q_1-Q_1$ discretization of the Stokes equations, using DWJ relaxation \cref{DWJ-Precondition-2Sweeps} with two sweeps of relaxation on the transformed pressure equation
\\
\begin{tabular}{|l||l|l|l|l|l|} \hline
\textbf{Method} & $p_1$ & $p_2$ & $p_3$ & $\mathbf{\rho_{\Psi_*}(\pb)}$ & $\mathbf{\sigma}(\pb)$ \\ \hline \hline
ROBOBOA-D, FDG-$t$=1e-6  & 1.443 & 1.124 & 1.283 & 0.333 & 1.596e-4 \\
ROBOBOA-DF & 1.248 & 1.280 & 1.110 & 0.333 & 0\\
HANSO-FI, FDG-$t$=1e-6, $\Ntheta^2=3^2$ & 1.321 & 1.357 & 1.174 & 0.406 & 2.129\\
HANSO-FI, FDG-$t$=1e-6, $\Ntheta^2=8^2$ & 1.196 & 1.264 & 1.072 & 0.347 & 0.727 \\
HANSO-FI, FDG-$t$=1e-6, $\Ntheta^2=9^2$ & 1.242 & 1.300 & 1.104 & 0.337 & 1.847 \\
\hline
\end{tabular}
\\
\raggedright \Cref{odd-compare-VK3}: Vanka relaxation with $n=3$ parameters
\\
\begin{tabular}{|l||l|l|l|l|l|} \hline
\textbf{Method} & $p_1$ & $p_2$ & $p_3$ & $\mathbf{\rho_{\Psi_*}(\pb)}$ & $\mathbf{\sigma}(\pb)$ \\ \hline \hline
ROBOBOA-D, FDG-$t$=1e-12  & 0.092 & 0.298 & 0.477 & 0.560 & 2.716  \\
ROBOBOA-D, FDG-$t$=1e-6  & 0.063 & 0.302 & 0.506 & 0.649 & 2.543  \\
ROBOBOA-DF & 0.218 & 0.293 & 0.465 & 0.437 & 0 \\
HANSO-FI, FDG-$t$=1e-12, $\Ntheta^3=3^3$ & 0.200 & 0.241 &  0.456 & 0.465 & 0.856\\
HANSO-FI, FDG-$t$=1e-6, $\Ntheta^3=3^3$ & 0.201 & 0.241 &  0.456 & 0.465 & 0.856 \\
\hline
\end{tabular}
\\
\raggedright \Cref{odd-compare-VK3}: Vanka relaxation with $n=5$ parameters
\\
\begin{tabular}{|l||l|l|l|l|l|l|l|} \hline
\textbf{Method} & $p_1$ & $p_2$ & $p_3$ & $p_4$ & $p_5$ & $\mathbf{\rho_{\Psi_*}(\pb)}$ & $\mathbf{\sigma}(\pb)$ \\ \hline \hline
ROBOBOA-D, FDG-$t$=1e-12  & 0.165 & 0.250 & 0.182 & 0.327 & 0.441 & 0.477 & 0.828  \\
ROBOBOA-D, FDG-$t$=1e-6  & 0.208 & 0.321 & 0.168 & 0.320 & 0.523 & 0.638 & 2.388  \\
ROBOBOA-DF & 0.218 & 0.321 & 0.202 & 0.248 & 0.457 & 0.459 & 0 \\
HANSO-FI, FDG-$t$=1e-12, $\Ntheta^3=3^3$ & 0.197 & 0.297 &  0.297 & 0.243 & 0.450 & 0.457 & 0.848\\
HANSO-FI, FDG-$t$=1e-6, $\Ntheta^3=3^3$ & 0.189 & 0.236 &  0.236 & 0.240 & 0.456 & 0.466 & 0.791\\
\hline
\end{tabular}
%
%
%
%
\\
\raggedright \Cref{odd-compare-control-problem}: 3D control problem \cref{Control-problem}
\\
\begin{tabular}{|l||l|l|l|l|} \hline
\textbf{Method} & $p_1$ & $p_2$ & $\mathbf{\rho_{\Psi_*}(\pb)}$ & $\mathbf{\sigma}(\pb)$ \\ \hline \hline
ROBOBOA-D, FDG-$t$=1e-12  & 0.913 & 0.734 & 0.908 & 0.125   \\
ROBOBOA-DF & 1.527 & 0.842 & 0.895 & 0 \\
HANSO-FI, FDG-$t$=1e-12, $\Ntheta^3=3^3$ & 0.668 & 0.842 &  0.895 & 1.023e-4 \\
\hline
\end{tabular}
\end{table}

\end{document}